\documentclass[12pt,reqno]{amsart}
\usepackage{a4wide}
\usepackage[T1]{fontenc}
\usepackage{amssymb,amsmath,amsthm,mathtools}
\usepackage{mathrsfs} 
\usepackage{enumerate}
\usepackage{enumitem}
\usepackage[dvipsnames]{xcolor}
\usepackage{graphicx}
\usepackage{tikz}
\usetikzlibrary{patterns}
\usepackage{hyperref}
\usepackage{amsrefs} 
\usepackage{cleveref} 
\hypersetup{colorlinks=true, linkcolor=MidnightBlue, citecolor=teal, urlcolor=cyan}

\makeatletter
\g@addto@macro\bfseries{\boldmath} 

\renewcommand{\tocsection}[3]{%
  \indentlabel{\@ifnotempty{#2}{\bfseries#1 #2.\;\,}}\bfseries#3}

\renewcommand{\tocsubsection}[3]{%
  \indentlabel{\@ifnotempty{#2}{#1 #2.\;\,}}#3}

\newcommand\@dotsep{3.5}
\def\@tocline#1#2#3#4#5#6#7{\relax
  \ifnum #1>\c@tocdepth 
  \else
    \par \addpenalty\@secpenalty\addvspace{#2}%
    \begingroup \hyphenpenalty\@M
    \@ifempty{#4}{%
      \@tempdima\csname r@tocindent\number#1\endcsname\relax
    }{%
      \@tempdima#4\relax
    }%
    \parindent\z@ \leftskip#3\relax \advance\leftskip\@tempdima\relax
    \rightskip\@pnumwidth plus1em \parfillskip-\@pnumwidth
    #5\leavevmode\hskip-\@tempdima{#6}\nobreak
    \leaders\hbox{$\m@th\mkern \@dotsep mu\hbox{.}\mkern \@dotsep mu$}\hfill
    \nobreak
    \hbox to\@pnumwidth{\@tocpagenum{\ifnum#1=1\bfseries\fi#7}}\par
    \nobreak
    \endgroup
  \fi}
\AtBeginDocument{%
\expandafter\renewcommand\csname r@tocindent0\endcsname{0pt}
}
\def\l@subsection{\@tocline{2}{0pt}{2.5pc}{5pc}{}}
\makeatother

\newtheorem{theorem}{Theorem}[section]
\newtheorem{lemma}[theorem]{Lemma}
\newtheorem{corollary}[theorem]{Corollary}
\newtheorem{proposition}[theorem]{Proposition}
\newtheorem{fact}[theorem]{Fact}
\newtheorem{claim}[theorem]{Claim}
\newtheorem*{claim*}{Claim}

\newtheorem{problem}[theorem]{Problem}
\newcounter{maintheorem}

\newtheorem{mainth}[maintheorem]{Theorem}
\Crefname{maintheorem}{Theorem}{Main Theorems}
\theoremstyle{remark}
\newtheorem{remark}[theorem]{Remark}
\theoremstyle{definition}
\newtheorem{definition}[theorem]{Definition}

\newtheorem{example}[theorem]{Example}
\numberwithin{equation}{section}
\makeatother

\newcommand{\R}{\mathbb{R}}
\newcommand{\ZZ}{\mathbb{Z}}
\newcommand{\N}{\mathbb{N}}
\newcommand{\e}{\varepsilon}
\renewcommand{\theta}{\vartheta}
\renewcommand{\rho}{\varrho}
\newcommand{\p}{\varphi}

\newcommand{\F}{\mathcal{F}}
\newcommand{\B}{\mathcal{B}}
\newcommand{\J}{\mathcal{J}}

\newcommand{\n}{\left\Vert\cdot\right\Vert}
\newcommand{\nn}[1]{{\left\vert\kern-0.25ex\left\vert\kern-0.25ex\left\vert #1 
\right\vert\kern-0.25ex\right\vert\kern-0.25ex\right\vert}}
\newcommand{\bone}{\text{\usefont{U}{bbold}{m}{n}1}}
\newcommand{\cut}{\mathord{\upharpoonright}}
\renewcommand{\leq}{\leqslant}
\renewcommand{\geq}{\geqslant}

\renewcommand{\hat}{\widehat}
\newcommand{\spn}{{\rm span}}
\DeclareMathOperator{\cf}{cf}
\DeclareMathOperator{\inte}{int}
\DeclareMathOperator{\supp}{supp}
\DeclareMathOperator{\dens}{dens}
\DeclareMathOperator{\dist}{dist}

\setlist[enumerate]{leftmargin=*, widest=iii}
\setlist[enumerate,1]{label=\rm{(\roman*)}, ref=\roman*}
\setlist[itemize]{leftmargin=*, widest=iii}

\newcommand{\V}{\mathcal{V}}
\newcommand{\X}{\mathcal{X}}
\newcommand{\Y}{\mathcal{Y}}
\newcommand{\W}{\mathcal{W}}
\newcommand{\Z}{\mathcal{Z}}
\newcommand{\D}{\mathcal{D}}
\renewcommand{\P}{\mathcal{P}}
\newcommand{\K}{\mathcal{K}}
\newcommand{\C}{\mathcal{C}}
\newcommand{\M}{\mathcal{M}}
\renewcommand\qedsymbol{$\blacksquare$} 

\title[Packings in classical Banach spaces]{Packings in classical Banach spaces}

\author[C.A.~De~Bernardi]{Carlo Alberto De Bernardi}
\address[C.A.~De~Bernardi]{Dipartimento di Matematica per le Scienze economiche, finanziarie ed attuariali, Universit\`a Cattolica del Sacro Cuore, 20123 Milano, Italy\newline
\href{https://orcid.org/0000-0002-9654-1324}{ORCID: \texttt{0000-0002-9654-1324}}}
\email{carloalberto.debernardi@unicatt.it, carloalberto.debernardi@gmail.com}

\author[T.~Russo]{Tommaso Russo}
\address[T.~Russo]{Universit\"{a}t Innsbruck, Department of Mathematics, Technikerstra\ss e 13, 6020 Innsbruck, Austria \newline
\href{https://orcid.org/0000-0003-3940-2771}{ORCID: \texttt{0000-0003-3940-2771}}}
\email{tommaso.russo@uibk.ac.at, tommaso.russo.math@gmail.com}

\author[\c S.~Sezgek]{\c Seyda~Sezgek}
\address[\c S.~Sezgek]{Department of Mathematics, Mersin University, Mersin, Turkey \newline
\href{https://orcid.org/0000-0001-9035-3114}{ORCID: \texttt{0000-0001-9035-3114}}}
\email{seydasezgek@mersin.edu.tr}

\author[J.~Somaglia]{Jacopo Somaglia}
\address[J.~Somaglia]{Politecnico di Milano, Dipartimento di Matematica, Piazza Leonardo da Vinci 32, 20133 Milano, Italy \newline
\href{https://orcid.org/0000-0003-0320-3025}{ORCID: \texttt{0000-0003-0320-3025}}}
\email{jacopo.somaglia@polimi.it}


\thanks{The research of the authors has been partially supported by the GNAMPA (INdAM -- Istituto Nazionale di Alta Matematica).}

\begin{document}
\begin{abstract} We obtain several new results on the simultaneous packing and covering constant $\gamma(\mathcal{X})$ of a Banach space $\mathcal{X}$, and its lattice counterpart $\gamma^*(\mathcal{X})$. These constants measure how efficient a (lattice) packing by unit balls in $\mathcal{X}$ can be, the optimal case being that $\gamma(\mathcal{X})= 1$ and the worst that $\gamma(\mathcal{X})= 2$. Our first main result is that $\gamma(\mathcal{X})> 1$ whenever $B_\mathcal{X}$ admits a LUR point, which leads us to a negative answer to a question of Swanepoel. We also develop general methods to compute these constants for a large class of spaces. As a sample of our findings:
\begin{enumerate}
    \item $\gamma^*(\mathcal{X})= 1$ when $\mathcal{X}$ is a separable octahedral Banach space, or $\mathcal{X}= \mathcal{C}(\mathcal{K})$, where $\mathcal{K}$ is zero-dimensional;
    \item $\gamma(\ell_p(\kappa)\oplus_r \mathcal{X})= \gamma^*(\ell_p(\kappa)\oplus_r \mathcal{X})= \frac{2}{2^{1/p}}$, whenever $\rm{dens}(\mathcal{X})< \kappa$ and $1\leq r\leq p< \infty$;
    \item $\gamma(L_p(\mu))= \gamma^*(L_p(\mu))= \frac{2}{2^{1/p}}$ for $1\leq p\leq 2$ and every measure $\mu$;
    \item there exist reflexive (resp.~octahedral) Banach spaces $\mathcal{X}$ with $\gamma(\mathcal{X})= 2$.
\end{enumerate}
We leave a large area open for further research and we indicate several possible directions.
\end{abstract}
\maketitle
\tableofcontents

\section{Introduction}
The study of optimal packings by balls is a classical topic in mathematics, with several applications to self-correcting codes, cryptography, granular materials, and phase transitions, just to name a few. We refer, \emph{e.g.}, to \cites{Lowen, Micciancio, Zong_Research} for some of these applications and to \cites{ConwaySloane, Rogers, Zong_sphere_book, Zong_Expo} for expository presentations of parts of the theory. The most well-studied notion of optimality is obtained by requiring the packing to have maximal density: given a convex body $C$ in $\R^n$, $\delta(C)$ is the maximal density of a packing by translates of $C$ and $\delta^*(C)$ is the maximal density of a lattice packing by $C$. In the particular case when $C$ is the Euclidean ball in $\R^n$, it is customary to just write $\delta_n$ and $\delta^*_n$. The exact value of $\delta_n$ is only known for $n=2$ \cite{Thue}, $n=3$ \cite{Hales} (the case $n=3$ was known as \emph{Kepler's problem} and was originally considered in order to stack cannonballs on a ship in the most efficient way), $n=8$ \cite{Viazovska}, and $n=24$ \cite{Viazovska24}. We refer to \cite{Cohn_notices} for a survey introduction to the cases $n=8,24$ and \cite{Viazovska_univ} for a generalisation of this result. There also is a large literature involving upper and lower bounds for $\delta_n$ and $\delta^*_n$, starting from the classical Minkowski--Hlawka theorem that $\delta^*_n\geq 2\zeta(n)2^{-n}$ till the very recent bound $\delta^*_n\geq c n^2 2^{-n}$, obtained by Klartag \cite{Klartag}. We refer to said paper and to \cites{CJMS, Sardari} for more references and information on upper bounds (the state-of-the-art upper bound is roughly $\delta_n\lesssim 0.66^n$). Concerning the constants $\delta(C)$ and $\delta^*(C)$ for arbitrary convex bodies $C$, we just mention here that
\[ \delta(C)= \delta^*(C)\geq \frac{2}{3} \]
for all convex bodies in $\R^2$, and equality holds if an only if $C$ is a triangle. We refer to \cite{Zong_Expo} and references therein for a reference on this fact and more results.

Rogers \cite{Rogers_JLMS} considered a different notion of optimality of a packing, by requiring that slightly inflating each body one obtains a covering of the space. More precisely, given a convex body $C$ in $\R^n$, $\gamma(C)$ is the smallest $r>0$ such that there exists a packing $\{C+\lambda\}_{\lambda\in \Lambda}$ by translates of $C$ such that $\{rC+\lambda\}_{\lambda\in \Lambda}$ covers $\R^n$; $\gamma^*(C)$ is defined by requiring additionally that $\Lambda$ is a lattice. (The fact that such a minimum exists follows from Mahler compactness theorem \cite{Mahler}, see, \emph{e.g.}, \cite{Coppel}*{Section VIII.6}.) In the particular case when $C$ is the Euclidean ball, these constants are indicated $\gamma_n$ and $\gamma^*_n$ respectively and called \emph{Rogers' constants}. As in the previous paragraph, these constants are only known for very few dimensions: $\gamma^*_n$ has been determined only for $n\leq 5$ and $\gamma_n$ only for $n\leq 3$ (we refer to \cite{Zong_Research}*{Table 3} or \cite{LiZong} for the exact values and references). It is easy to check that $(\gamma_n)^n\delta_n\geq 1$; therefore the upper bound discussed in the previous paragraph yields $\gamma_n\geq 2^{0,599} +o(1)$. Interestingly, it follows that $\liminf \gamma_n > \sqrt{2}$ (which, as we will see below, is the value of $\gamma$ for infinite-dimensional Hilbert spaces). Moreover, Butler \cite{Butler} proved that $\gamma^*(C)\leq 2+o(1)$ for every symmetric convex body in $\R^n$ (which improved the estimate $\gamma^*(C)\leq 3$ from \cite{Rogers_JLMS}). In low dimensions there of course are more precise estimates. For instance, for every convex body $C$ in the plane one has
\[ \gamma(C)= \gamma^*(C)\leq \frac{3}{2} \]
with equality if and only if $C$ is a triangle \cite{Linhart} and 
\[ \gamma^*(C)\leq 2(2-\sqrt{2}) \]
for every symmetric convex body, with equality if and only if $C$ is an affinely regular octagon \cite{Zong_Adv} (see, \emph{e.g.}, \cite{Zong_Expo}*{Section 6.1}). In dimension $3$, one has $\gamma^*(O)= \frac{7}{6}$ where $O$ is the regular octahedron (namely, the unit ball of the $\ell_1$-norm in $\R^3)$ \cite{LiZong} and $\gamma^*(C)\leq \frac{7}{4}$ for all convex bodies \cite{Zong7/4}. One more important result is that, for every convex body $C$ in $\R^n$, $\gamma(C)=1$ if and only if $\gamma^*(C)=1$, \cite{Zong_strange_book}*{Chapter 3} (but it is not known if $\gamma(C)= \gamma^*(C)$ for all convex bodies in $\R^n$, see \Cref{probl: gamma= gamma*}).

While the definition of $\delta(C)$ is inherently finite dimensional, defining $\gamma(C)$ in infinite dimensions only requires the adjustment of specifying which normed space $\X$ the body $C$ belongs to. Further, if one restricts the attention to symmetric bodies, considering symmetric bodies in $\X$ is equivalent to considering unit balls of equivalent norms on $\X$. In other words, it is sufficient to investigate $\gamma(B_\X)$ for every normed space $\X$. As in \cite{Swanepoel}, we actually associate the constant $\gamma$ to the normed space rather than to its unit ball; thus, we just write $\gamma(\X)$ instead of $\gamma(B_\X)$ (and, of course, similarly for $\gamma^*(\X)$). We refer to \Cref{sec: def gamma} for the precise definition of the constants and their basic properties.

We now summarise the known results on $\gamma(\X)$ and $\gamma^*(\X)$ for infinite-dimensional normed spaces $\X$. The first known result is the fact that $\gamma^*(\X)\leq 3$ for every infinite-dimensional normed space $\X$, due to Rogers \cite{Rog84}. This was recently improved to $\gamma^*(\X)\leq 2$ in \cite{DRShilbert}, which solved a problem due to Swanepoel \cite{Swanepoel} (the case of separable $\X$ was already answered in \cite{DOSZ}). Casini, Papini, and Zanco \cite{CPZ} obtained the lower bound $\gamma(\X)\geq \frac{2}{K(\X)}$ (\cite{CPZ} states the result for reflexive spaces only, but essentially the same argument works for every normed space, see \Cref{prop: CPZ gamma>=2/K} below). Summarising, we have
\[ 1\leq \frac{2}{K(\X)}\leq \gamma(\X)\leq \gamma^*(\X)\leq 2 \]
for every infinite-dimensional normed space ($K(\X)$ is the Kottman constant of $\X$, whose definition we recall in \Cref{sec: Kottman}). In few cases, exact computations are also available. To begin with, it is clear that $\gamma^*(c_0)= \gamma^*(\ell_\infty)=1$, just by consideration of the even integers lattice. Swanepoel \cite{Swanepoel} proved that $\gamma(\ell_p)= 2/2^{1/p}$ for $1\leq p< \infty$, which was recently improved in \cite{DRShilbert}*{Corollary 3.4} to $\gamma^*(\ell_p(\kappa))= 2/2^{1/p}$ for $1\leq p< \infty$ and every infinite cardinal $\kappa$. In each of the previous examples, $\gamma(\X)$ is actually equal to $\frac{2}{K(\X)}$, which naturally raises the question whether equality might hold for all normed spaces. 

\begin{problem}[Swanepoel, \cite{CPZ}*{Question~1.10}, \cite{Swanepoel}*{Section 6.2}]\label{probl: gamma=2/k} Is it true that
\[ \gamma(\X)= \frac{2}{K(\X)} \]
for every infinite-dimensional Banach space $\X$?
\end{problem}

Our original motivation for the research presented in this paper was to find a counterexample to this problem. As it turns out, throughout the paper we shall offer a large collection of counterexamples, including in particular reflexive ones (the reflexive case was specifically asked in \cite{CPZ}*{Question~1.10}). Our first and main source of examples is obtained via the following result, that will be the main focus of  \Cref{sec: packing LUR}.

\begin{mainth}\label{mth: LUR point} Every normed space $\X$ such that $B_\X$ has a LUR point satisfies $\gamma(\X)>1$. As a consequence, every infinite-dimensional normed space is isomorphic to a normed space $\Y$ with $\gamma(\Y)>1$ and $K(\Y)=2$.
\end{mainth}

Interestingly, there are examples of rotund, octahedral normed spaces; see, \emph{e.g.}, \cite{DPS_MLUR}*{Example~3.12} for an example based on a renorming of $\ell_1$. As a consequence of our results (\Cref{mth: gamma}\eqref{mth.i: octa} below), such spaces satisfy $\gamma^*(\X) =1$. Therefore, the assumption on the LUR point cannot be replaced by, say, mere rotundity of the unit ball.

It follows from a (more general) result of Vesel\'y and the first-named author \cite{DEVEtiling}*{Theorem~4.9} that a normed space $\X$ such that $B_\X$ admits a LUR point does not admit a tiling by balls of radius one. Our result offers a generalisation of this fact, as the condition that $\gamma(\X)>1$ means that all packings of $\X$ are `far' from being a tiling. Here we should perhaps notice that every normed space $\X$ that admits a tiling by unit balls clearly satisfies $\gamma(\X) =1$, but the converse does not hold, as we mention below.

Once it has been clarified that the computation of the packing constants $\gamma(\X)$ and $\gamma^*(\X)$ cannot be simply reduced to that of $K(\X)$, we face the essentially unexplored line of investigation to determine these constants for infinite-dimensional normed spaces. In the second, and main, part of the paper, we present several results in this direction, that provide information on all classical Banach spaces, in most cases giving exact computations. An outline of our most significant contributions is given in the following theorem.

\begin{mainth}\label{mth: gamma} Let $\X$ be a normed space and $\kappa$ be an infinite cardinal.
\begin{enumerate}[label=\rm{(\Roman*)}, ref=\Roman*]
    \item\label{mth.i: octa} $\gamma^*(\X)=1$ whenever $\X$ is a separable octahedral normed space (\Cref{thm: octahedral}), or $\X=\C(\K)$ where $\K$ is zero-dimensional (\Cref{thm: C(K)}).
    \item\label{mth.i: ell_p} If $\dens(\X)< \kappa$, then
    \[ \gamma(\ell_p(\kappa) \oplus_r \X)= \gamma^*(\ell_p(\kappa) \oplus_r \X)= \frac{2}{2^{1/p}} \]
    for every $1\leq r\leq p< \infty$ (\Cref{thm: gamma ell_p}).
    \item\label{mth.i: L_p} If $(\M, \Sigma, \mu)$ is any measure space and $\dens(\X)< \dens(L_p(\mu))$,
    \[ \min\left\{\frac{2}{2^{1/p}}, \frac{2}{2^{1/q}}\right\}\leq \gamma(L_p(\mu)\oplus_r \X)\leq \gamma^*(L_p(\mu)\oplus_r \X)\leq \frac{2}{2^{1/p}}, \]
    for every $1\leq r\leq p< \infty$, where $q$ is the conjugate exponent of $p$ (\Cref{thm: gamma L_p}).
    \item\label{mth.i: UC} If $\X$ is a super-reflexive space of density $\kappa$
    \[ \frac{1}{1-\delta_\X(1)}\leq \frac{2}{K(\X;\kappa)}\leq \gamma(\X)\leq \gamma^*(\X)\leq \frac{2}{1+\p_\X(1)}, \]
    where $\delta_\X$ is the modulus of convexity and $\p_\X$ is the tangential modulus of convexity, defined in \Cref{def: phi_X} (\Cref{thm: superreflexive and not}).
    \item\label{mth.i: gamma =2} There exist Banach spaces $\X$ (of density $\omega_\omega)$ such that $\gamma(\X)=2$. Additionally, such spaces can be taken to be reflexive, or octahedral (\Cref{thm: gamma=2}).
\end{enumerate}
\end{mainth}

We now comment on how this result provides information on all classical Banach spaces. To begin with, \eqref{mth.i: ell_p} generalises the results from \cites{DRShilbert, Swanepoel} that determined $\gamma(\ell_p(\kappa))$ and $\gamma^*(\ell_p(\kappa))$. Further, it also leads to more counterexamples to \Cref{probl: gamma=2/k} (\Cref{ex: octa gamma>1}). The analogue result \eqref{mth.i: L_p} for all $L_p(\mu)$ spaces leads us to the precise computation of the constants only in the case when $p\leq 2$, because when $p>2$ the Kottman constant of $L_p(\mu)$ equals $2^{1/q}$ if $\mu$ is not purely atomic; thus, the lower bound is weaker in this case. Concerning $\C(\K)$ spaces, it follows in particular from \eqref{mth.i: octa} that $\gamma^*(\C(\K)) =1$ when the compact $\K$ is zero-dimensional (in particular, scattered) or a perfect metric space. Hence, this applies to, \emph{e.g.}, $\C([0,\alpha])$ for all ordinals $\alpha$, $\C([0,1])$, and $\C(2^\omega)$.

While these results cover $\C(\K)$ and $L_p(\mu)$ spaces, namely all classical spaces, it is to be remarked that all these results are intrinsically isometric and none of them applies to renorming of these spaces. In other words, the renorming theory is almost completely unexplored, with only two results known. The first is \Cref{prop: renorm ell_2} showing that there is a norm $\nn\cdot$ on $\ell_2$ such that $\gamma^*((\ell_2,\nn\cdot)) =1$; however, we don't know if it is possible to achieve $\gamma^*((\ell_2,\nn\cdot)) >\sqrt{2}$ (\Cref{probl: renorm ell2}). The second is that it is always possible to get $\gamma(\X) >1$, by adding a LUR point in $B_\X$, by \Cref{mth: LUR point}.

Moreover, \eqref{mth.i: octa} yields a large class of Banach spaces that satisfy $\gamma^*(\X) =1$, beside the cases of $\ell_1(\kappa)$ known from \cite{DRShilbert} and $\C([0,1])$ mentioned above. In fact, standard examples of octahedral Banach spaces also include, for instance, $L_1$, $\C(\K)$ when $\K$ is perfect, and all Banach spaces with the Daugavet property. Further, octahedrality has been studied in Lipschitz-free spaces \cite{octa_free} and their duals \cite{octa_Lip}, in spaces of operators \cite{octa_operator}, or in tensor products \cite{octa_tensor}. For information on octahedral Banach spaces, we refer, \emph{e.g.}, to \cites{octa_slice, octa_bidual, octa_Daugavet}. As it follows from \eqref{mth.i: gamma =2}, the separability assumption in the octahedral case of \eqref{mth.i: octa} is essential; there even are octahedral spaces of density $\omega_1$ for which $\gamma(\X) >1$, such as $\ell_1\oplus_1 \ell_p(\omega_1)$ (\Cref{ex: octa gamma>1}). However, while, for the sake of simplicity, we only mentioned octahedrality in \eqref{mth.i: octa}, most of our results actually concern its generalisation to larger cardinals, namely $(<\kappa)$-octahedrality, introduced in \cite{CLL}, which gives rise to more examples (\Cref{rmk: <kappa octa examples}). For instance, as we prove in \Cref{prop: L_p p-octa}, every $L_1(\mu)$ space of density $\kappa$ is $(<\kappa)$-octahedral (this is possibly a folklore fact, but no proof appears in the literature).

Among all the results in \Cref{mth: gamma}, the one that we consider to be the most interesting and surprising is \eqref{mth.i: gamma =2}. In fact, the validity of the inequality $\gamma(\X)\leq 2$ for all normed spaces is just obtained by considering any maximal packing, without any need for geometric considerations. Similarly, the fact that $K(\X)\geq 1$ for all infinite-dimensional spaces is just a direct application of Zorn lemma (for each $\e>0$, a maximal $(1-\e)$-separated set in $B_\X$ has cardinality $\dens(\X)$). In the case of Kottman constant, a deep result of Elton and Odell \cites{EltonOdell} is that $K(\X) >1$ for all infinite-dimensional spaces. From this perspective, it would have been tempting to speculate that careful geometric arguments, likely combined with combinatorial considerations, could prove that $\gamma(\X)< 2$ for all spaces. Because of the inequality $\gamma(\X)\geq \frac{2}{K(\X)}$, this would have actually been an improvement of the Elton--Odell theorem. Incidentally, by the very same argument, proving that $\gamma(\X)< 2$ for all separable spaces $\X$ (\Cref{probl: separable gamma=2}) would yield a novel proof of the Elton--Odell theorem (and because of the same inequality, spaces as in \eqref{mth.i: gamma =2} are again counterexamples to \Cref{probl: gamma=2/k}). However, \eqref{mth.i: gamma =2} shows that there exist spaces where it is impossible to produce packings that are more dense than the `random' packing offered by Zorn lemma. A different perspective on this fact is that in spaces as in \eqref{mth.i: gamma =2} each packing can be modified into an optimal one, just by adding more balls. Intriguingly, such spaces can also be taken to be octahedral, or reflexive; however, we don't know if there are super-reflexive such spaces (\Cref{probl: separable gamma=2}) and it follows from \eqref{mth.i: UC} that there aren't uniformly convex ones.

We should also remark that, in a few cases, namely $\C(\K)$ spaces with $\K$ extremally disconnected (\Cref{thm: C(K)}) and $L_\infty(\mu)$ spaces (\Cref{prop: L_infty}), our argument proving that $\gamma^*(\X) =1$ even yields a lattice tiling by balls. However, differently from the finite-dimensional case, the assumption that $\gamma(\X) =1$ (resp.~$\gamma^*(\X) =1$) is weaker than the presence of a (lattice) tiling by balls of radius one, as we shall discuss in \cite{DRSS_ell1}.

To conclude our explanation of the novelty of \Cref{mth: gamma}, we stress that while our results significantly advance the state of the art and yield information on all classical Banach spaces, it is plain that they leave a wide area for further investigation. In \Cref{sec: problem} we indicate a sample of natural problems that stem from our results.

We now describe the strategy to prove \Cref{mth: gamma} and the organisation of the second part of the paper. As it turns out, all our results in \Cref{mth: gamma} follow from two general methods to estimate the constants $\gamma(\X)$ and $\gamma^*(\X)$ and we consider the distillation of these methods to be the most significant outcome of our paper. All the upper bounds we obtain follow from a general procedure to construct discrete subgroups of normed spaces, that largely generalises the construction performed in \cite{DRShilbert} to show that $\gamma^*(\X)\leq 2$ for all infinite-dimensional normed spaces. In turn, this argument borrowed an idea from Klee's seminal paper \cite{Klee_MathAnn}, which has influenced several other results, such as those in \cites{Klee2, DEVEtiling, KleeMalZan, Swanepoel}. We shall present this general construction in \Cref{sec: subgroup}, where we also prove \Cref{mth: gamma}\eqref{mth.i: octa}, which only requires this upper bound. 

Instead, the general lower bound, that we present in \Cref{sec: phi-octahedral}, requires the introduction of the notion of $\phi$-octahedral normed space, where $\phi$ is a modulus. Intuitively speaking, we replace the linear growth in the `orthogonal' direction with the prescription of a growth lower bounded by the modulus $\phi$. While the definition is formally very similar to that of octahedrality (which gave us the inspiration for the terminology), the properties of $\phi$-octahedral normed space can differ quite significantly from octahedral ones; to wit, uniformly convex spaces are $\phi$-octahedral (\Cref{thm: UC implies phi-octa}). After some basic results on $\phi$-octahedral normed space and some results on stability under direct sums, the main result of the section is \Cref{thm: gamma of phi-octa}, where we prove the above-mentioned lower bound. Beside the crucial role that the result bears for the paper, it intrinsic interest also stems from the fact that it generalises the result of Casini, Papini, and Zanco \cite{CPZ} that we mentioned above (see also \Cref{rmk: CPZ for kappa}).

While \Cref{sec: phi-octahedral} also contains the proof of \Cref{mth: gamma}\eqref{mth.i: gamma =2}, our applications to uniformly convex and, in particular, $L_p(\mu)$ spaces are presented in the subsequent \Cref{sec: UC and Leb}. Among others, there we introduce a variant $\p_\X$ of a modulus of convexity studied by Milman, which we require in order to prove that uniformly convex spaces $\X$ are $\p_\X$-octahedral. Finally, the last section of the paper (\Cref{sec: problem}) is dedicated to the discussion of some possible directions that is natural to consider after our research.

\section{Preliminaries}\label{sec: prelim}
For a real normed space $\X$, $B_\X$ and $S_\X$ denote the closed unit ball and the unit sphere of $\X$ respectively. We denote by $B(x,r)$ the closed ball with radius $r>0$ and centre $x$. The \emph{density character} (or just \emph{density)} $\dens(\X)$ of $\X$ is the smallest cardinality of a set with dense span in $\X$. In particular, which is perhaps not entirely standard, for us $\dens(\X)< \omega$ means that $\X$ is finite dimensional. The cardinality of a set $S$ is indicated by $|S|$. We regard cardinal numbers as initial ordinal numbers; hence, we write $\omega$ for the first infinite cardinal $\aleph_0$ and $\omega_n$ for the cardinal $\aleph_n$ ($n\geq 1$).

\subsection{Convexity, packings, and moduli}\label{sec: convex}
A \emph{convex body} is a closed convex set with non-empty interior. A \emph{packing} in $\X$ is a collection of mutually non-overlapping convex bodies (two convex bodies are \emph{non-overlapping} if they have disjoint interiors). A \emph{tiling} is a packing that is also a covering for $\X$. A packing (resp.~a tiling) $\F$ is \emph{lattice} if there are a convex body $C$ and a discrete subgroup (sometimes called a lattice) $\Lambda$ in $\X$ such that $\F= \{C+\lambda\}_{\lambda\in \Lambda}$. Throughout the paper, we will only consider packings or tilings by translates of $B_\X$. A point $x\in \X$ is a \emph{singular point} for a family $\F$ of convex bodies if every neighbourhood of $x$ intersects infinitely many elements of $\F$. 

For a normed space $\X$, the \emph{modulus of convexity} $\delta_\X$ is defined by
\[ \textstyle \delta_\X(\e)\coloneqq \inf\left\{1-\left\|\frac{x+y}2 \right\|\colon\, x,y\in B_\X,\, \|x-y\|\geq\e \right\}, \qquad \e\in[0,2]. \]
Moreover, given $x_0\in S_\X$, the \emph{modulus of local uniform rotundity} at $x_0$ is defined as 
\[ \textstyle \delta_\X(x_0,\e)\coloneqq \inf\left\{1-\|\frac{x_0+y}2 \|\colon\, y\in B_\X,\, \|x_0-y\|\geq\e \right\}, \qquad \e\in[0,2]. \]
Then, $\X$ is uniformly convex if and only if $\delta_\X (\e)>0$ for each $\e\in (0,2]$, and $x_0$ is a LUR (locally uniformly rotund) point for $B_\X$ if and only if $\delta_\X (x_0,\e)>0$ for each $\e \in(0,2]$. It is well-known that the value of $\delta_\X(\e)$ is not affected if one only considers $x,y$ such that $x,y\in S_\X$ and/or $\|x-y\|=\e$; further, the function $\e\mapsto \delta_\X(\e)$ is continuous on $[0,2)$. Perhaps less well-known is the fact that a similar statement holds also for the modulus of local uniform rotundity, \cite{danes}: in the definition of $\delta_\X(x_0,\e)$ one can equivalently consider $y\in S_\X$ and/or $\|x-y\|=\e$. As a consequence of this fact, we readily deduce that $\delta_\X(x_0,\e)\leq \e/2$ for all $x_0\in \X$ and $\e>0$. For more information on these moduli, we refer, \emph{e.g.}, to \cite{LiTzaII}; see also \cites{DEVEmoduli, DEVEmodulilocal} for an extension of the previous results to convex bodies. Let us also recall here two important results that we require. The first is Nordlander inequality \cite{Nordlander} that, for all normed spaces of dimension at least $2$,
\[ \delta_\X(t)\leq 1- \sqrt{1-t^2/4}, \]
with equality if and only if $\X$ is a Hilbert space. The second, due to James \cite{James}, is that $\X$ is uniformly non-square, hence super-reflexive, if $\delta_\X(\e)>0$ for some $\e<2$.

\subsection{Kottman constants}\label{sec: Kottman}
A subset $\P$ of a normed space $\X$ is \emph{$r$-separated} if $\|x-y\|\geq r$ for all distinct $x,y\in \P$. The \emph{Kottman constant} $K(\X)$ \cite{Kottman} of $\X$ is
\[ K(\X)\coloneqq \sup\{r>0\colon B_\X \text{ contains an infinite $r$-separated set} \}. \]

By the celebrated Elton--Odell theorem \cite{EltonOdell}, $K(\X)>1$ for any normed space $\X$. Further, $K(\X)=2$ if $\X$ contains a subspace isomorphic to $c_0$ or $\ell_1$ (by James' distortion theorem), and $K(\ell_p(\kappa))= 2^{1/p}$, for every infinite cardinal $\kappa$ and every $1\leq p< \infty$. Another standard fact is that the Kottman constant of a finite $\ell_p$-sum of normed spaces is the maximum of the Kottman constants of the factors. For references on these facts and more information, we refer, \emph{e.g.}, to \cites{CGKP, CGP, HKR_JFA, R_RACSAM}. For our results, it will also be important to quantify separation of uncountable subsets of $B_\X$. Therefore, we introduce the following definition.

\begin{definition} Let $\X$ be a normed space and $\kappa$ a cardinal number. The \emph{$\kappa$-Kottman constant} $K(\X;\kappa)$ of $\X$ is
\[ K(\X; \kappa)\coloneqq \sup\{r>0\colon B_\X \text{ contains an $r$-separated set of cardinality }\kappa \}. \]
\end{definition}

For information on uncountable separated sets, we refer to \cites{DESOVEstar, HKR_TAMS, Kosz_Jussieu, Kosz_TAMS} and references therein. The following basic facts will be needed in what follows.

\begin{fact}\label{fact: K with a small summand} Let $\kappa$ be an infinite cardinal and $\X$ and $\Y$ be normed spaces, where $\dens(\Y)< \cf(\kappa)$. Then, $K(\X\oplus_p \Y;\kappa)= K(\X;\kappa)$ for all $p\in[1,\infty]$.
\end{fact} 

\begin{proof} Clearly, we only need to prove the `$\leq$' inequality. Let $((x_\alpha,y_\alpha))_{\alpha< \kappa}$ be an injective enumeration of an $r$-separated subset $\P$ of $B_{\X\oplus_p \Y}$. Fix $\e>0$ and write $\Y$ as a union of $\dens(\Y)$-many balls of radius $\e$. Since $\dens(\Y)< \cf(\kappa)$, one such ball must contain $\kappa$-many vectors $y_\alpha$. In other words, up to replacing $\P$ with a subset, still of cardinality $\kappa$, we can assume that $\|y_\alpha- y_\beta\|\leq \e$ for all $\alpha,\beta< \kappa$. Thus, $(x_\alpha)_{\alpha< \kappa}$ is $(r-2\e)$-separated in $B_\X$, whence $r-2\e\leq K(\X;\kappa)$. As $\e$ and $\P$ are arbitrary, we get $K(\X\oplus_p \Y;\kappa)\leq K(\X;\kappa)$.
\end{proof}

\begin{fact}\label{fact: K d_BM cont} For all isomorphic normed spaces $\X$ and $\Y$ and infinite cardinals $\kappa$
\[ K(\X;\kappa)\leq d_{BM}(\X,\Y)\cdot K(\Y;\kappa). \]
\end{fact}

\begin{proof} The proof is just the same as for the countable case, which is classical, \cite{Kottman}. However, it is so short that we give it here. Suppose that $\P\subseteq B_\X$ is an $r$-separated set of cardinality $\kappa$ and $T\colon \X\to \Y$ is such that $\|x\|\leq \|Tx\|\leq M\|x\|$. Then, $T(\P)$ is $r$-separated in $M\cdot B_\Y$, hence $r\leq M\cdot K(\Y; \kappa)$. Since $\P$ and the isomorphism $T$ are arbitrary, the result follows. 
\end{proof}

\subsection{The simultaneous packing and covering constant}\label{sec: def gamma}
Since we only consider packings by balls of radius $1$, the condition that the balls in the packing are non-overlapping is equivalent to requiring the centers to be $2$-separated. Further, the condition from the Introduction that inflating the balls by $r$ yields a covering just means that the set of centers is $r$-dense (a set $\P\subseteq \X$ is \emph{$r$-dense} if for each $x\in \X$ there is $p\in \P$ with $\|x- p\|\leq r$, namely, the balls $\{B(p,r) \colon p\in \P\}$ cover $\X$). Hence, the following definition is equivalent to the one from the Introduction.

\begin{definition} The \emph{simultaneous packing and covering constant} $\gamma(\X)$ of $\X$ is 
\[ \gamma(\X)\coloneqq \inf\{r>0\colon \text{there exists a $r$-dense and $2$-separated set } \P\subseteq \X\}. \]

The \emph{lattice simultaneous packing and covering constant} $\gamma^*(\X)$ of $\X$ is
\[ \gamma^*(\X)\coloneqq \inf\{r>0\colon \text{there exists a $r$-dense and $2$-separated subgroup } \P\subseteq \X\}. \]
\end{definition}

Plainly, $\gamma(\X)\geq 1$. Since every maximal $2$-separated set in $\X$ is $2$-dense, Zorn's lemma implies that $\gamma(\X)\leq 2$. This was recently improved to $\gamma^*(\X)\leq 2$ in \cite{DRShilbert}*{Theorem~C(ii)}. Notice that some authors, for instance \cite{Swanepoel}, define packings as collections of mutually disjoint balls; this however doesn't affect the above definition, due to the infimum. Another interpretation is that $\gamma(\X)-1$ is the radius of the largest circular hole present in every packing of $\X$, \cite{Zong_deep}.

In a few instances, especially when constructing discrete subgroups, we find it more convenient to optimise the separation, rather than the density, parameter. In these circumstances we will use the following equalities.
\begin{equation}\label{eq: 2/gamma(X)}
\begin{split}
    \frac{2}{\gamma(\X)}=& \sup\{r>0\colon \text{there exists a $1$-dense and $r$-separated set } \P\subseteq \X \}\\
    \frac{2}{\gamma^*(\X)}=& \sup\{r>0\colon \text{there exists a $1$-dense and $r$-separated subgroup } \P\subseteq \X \}.
\end{split}
\end{equation}

We now give a couple of basic inequalities that we require at various places.
\begin{fact}\label{fact: gamma and d_BM} For all normed spaces $\X$ and $\Y$ one has
\[ \gamma(\X\oplus_\infty \Y)\leq\max\{\gamma(\X), \gamma(\Y)\} \qquad\text{and}\qquad \gamma^*(\X\oplus_\infty \Y)\leq\max\{\gamma^*(\X), \gamma^*(\Y)\}. \]
Further, if $\X$ and $\Y$ are isomorphic,
\[ \gamma(\Y)\leq d_{BM}(\X,\Y)\cdot \gamma(\X) \qquad\text{and}\qquad \gamma^*(\Y)\leq d_{BM}(\X,\Y)\cdot \gamma^*(\X). \]
\end{fact}
Observe that the first claim does not hold for different $\ell_p$-sums: for instance, $\gamma(\ell_1^3)= 7/6$ by \cite{LiZong}, while $\gamma(\ell_1^2)= \gamma(\R)= 1$.

\begin{proof} For the first part, notice that if $\P$ and $\mathcal{Q}$ are $r$-dense and $2$-separated in $\X$ and $\Y$ respectively, then $\P\times \mathcal{Q}$ is $r$-dense and $2$-separated in $\X\oplus_\infty \Y$. For the second part, if $T\colon \X\to \Y$ is such that $\|x\|\leq \|Tx\|\leq M\|x\|$ and $\P$ is $r$-dense and $2$-separated in $\X$, then $T(\P)$ is $Mr$-dense and $2$-separated in $\Y$.
\end{proof}

\begin{remark} Most of our arguments are of geometric flavour and don't require completeness. In the few instances where completeness is needed (such as \Cref{prop: CPZ gamma>=2/K} below), we consider the completion $\hat{\X}$ of $\X$ and rely on the fact that most parameters of normed spaces (\emph{e.g.}, $\delta_\X$, $K(\X;\kappa)$, $\gamma(\X)$) are invariant under taking completions. Notice however that we don't know if $\gamma^*(\X)= \gamma^*(\hat{\X})$ for all normed spaces $\X$.
\end{remark}

Casini, Papini, and Zanco showed in \cite{CPZ} that $\gamma(\X)\geq 2/K(\X)$ if $\X$ has an infinite-dimensional reflexive subspace and asked whether the same inequality holds true for every infinite-dimensional Banach space, \cite{CPZ}*{Question 1.3}. The argument in \cite{CPZ} depends on Corson's theorem \cite{Corson} and it turns out that the same argument carries over for all normed spaces, upon replacing the usage of Corson's theorem with its generalisation \cite{FZ06}. Thus, we have the following fact, whose short proof we give for the sake of completeness.

\begin{proposition}\label{prop: CPZ gamma>=2/K} For every infinite-dimensional normed space $\X$ one has
\[ \gamma(\X)\geq \frac{2}{K(\X)}. \]
\end{proposition}

\begin{proof} Notice that, if $\hat{\X}$ denotes the completion of $\X$, then $\gamma(\hat{\X})= \gamma(\X)$ and $K(\hat{\X})= K(\X)$. Thus, we can assume that $\X$ is a Banach space (which we require in order to apply \cite{FZ06}). If $\X$ contains an isomorphic copy of $c_0$, then $K(\X)=2$, and the desired inequality is obvious. Otherwise, we use \eqref{eq: 2/gamma(X)} to show that $2/\gamma(\X)\leq K(\X)$. Take a set $\P\subseteq \X$ that is $1$-dense and $r$-separated. Then, the collection $\B\coloneqq \{x+ B_\X\colon x\in \P\}$ is a covering of $\X$, which by \cite{FZ06} has a singular point $x_0$. Hence, for every $\e>0$, the ball $B(x_0,\e)$ intersects infinitely elements of $\B$; thus, $\|x_0-x\|\leq 1+\e$ for infinitely many elements $x$ of $\P$. In other words, the ball $B(x_0,1+\e)$ contains an infinite $r$-separated set, whence $r\leq (1+\e)\cdot K(\X)$. Letting $\e\to 0$, we deduce that $r\leq K(\X)$ and the conclusion follows from \eqref{eq: 2/gamma(X)}.
\end{proof}

\begin{remark}\label{rmk: CPZ for kappa} It is of course natural to wonder if, in a normed space $\X$ of density $\kappa$, one might improve the above inequality to
\[ \gamma(\X)\geq \frac{2}{K(\X; \kappa)}. \]
However, this in general fails to hold. For instance, $\gamma(c_0(\omega_1))=1$, while $K(c_0(\omega_1); \omega_1)= 1$, \cite{EltonOdell}*{Remarks (2)}. Since the canonical norm on $c_0(\omega_1)$ can be approximated by norms that are simultaneously Fr\'echet smooth and LUR \cite{DGZ}*{Corollary II.7.8} and the above inequality is continuous in the Banach--Mazur distance, it even fails to hold for spaces that are Fr\'echet smooth and LUR. On the other hand, our \Cref{thm: gamma of phi-octa} will give a sufficient condition for its validity, which in particular applies to all super-reflexive Banach spaces (\Cref{thm: superreflexive and not}).
\end{remark}

\section{LUR points and \texorpdfstring{\Cref{probl: gamma=2/k}}{Problem 1.1}}\label{sec: packing LUR}
The objective of this section is the proof of \Cref{mth: LUR point}. In particular, the main result of the section is that that $\gamma(\X)>1$ whenever $B_\X$ admits a LUR point (\Cref{thm: gamma LUR}). We then apply this result to derive our first counterexamples to Swanepoel's problem, \Cref{probl: gamma=2/k}. We start by mentioning two ingredients that we shall exploit in the proof of \Cref{thm: gamma LUR}. The first lemma can be found in \cite{DESOVEstar}*{Lemma 3.7}, or \cite{DEVEtiling}*{Lemma 4.5}.

\begin{lemma}[\cites{DESOVEstar, DEVEtiling}]\label{lemma: 3LUR balls} Let $\X$ be a normed space, $\eta\geq0$, and $B_0, B_1, B_2$ be three mutually non-overlapping balls of radius one, where $B_0\coloneqq B_\X$. Suppose that there exist points $x_i\in\partial B_i$ ($i=0,1,2$) such that ${\rm diam} \{x_0, x_1, x_2\}\leq \eta$; then
\begin{equation}\label{eq: 3balls}
	\textstyle {\rm diam}\bigl\{y\in S_\X\colon \|x_0+y\|\geq 2-\eta\bigr\} \geq2-2\eta.
\end{equation}
\end{lemma}

\begin{fact}\label{fact: protected-lur} Let $B$ and $C$ be convex bodies in $\X$. Assume that $x_0\in\partial B$ is an extreme point of $B$ and that $x_0\in \inte(B\cup C)$. Then, $x_0\in\inte C$.
\end{fact}

\begin{proof} Let $\e>0$ be such that $B(x_0,\e) \subseteq B\cup C$ and assume, towards a contradiction, that $x_0\notin \inte C$. By the Hahn-Banach theorem there exists $f\in S_{\X^*}$ such that $f(x)\leq f(x_0)$ for all $x\in C$. Then, $\{x\in B(x_0,\e) \colon  f(x)> f(x_0)\}\subseteq B$, so $\{x\in B(x_0,\e) \colon  f(x)\geq f(x_0)\}\subseteq B$ too (as $B$ is closed). This contradicts $x_0$ being an extreme point for $B$.
\end{proof}

We are now ready for the main result of the section.
\begin{theorem}\label{thm: gamma LUR} If $\X$ is a normed space such that $B_\X$ admits a LUR point, then $\gamma(\X)>1$.
\end{theorem}

\begin{proof} The rough idea is that, if $\gamma(\X)=1$, close to a LUR point $x_0$ of $B_\X$ there has to be a second ball $B_1$ of the packing. Since $x_0$ is LUR, the balls $B_\X$ and $B_1$ aren't enough to cover a large enough fraction of a neighbourhood of $x_0$ (\Cref{claim: 3rd close ball}). Hence, there exists a third ball $B_2$ close to $x_0$, which then contradicts \Cref{lemma: 3LUR balls}. Now to the details.

Let $x_0$ be a LUR point for $B_0\coloneqq B_\X$; by definition of LUR point, 
\[ {\rm diam}\big\{y\in S_\X\colon \|x_0+y\|\geq 2-\eta\big\} \to 0, \qquad \text{as } \eta\to 0; \]
hence, there is $\eta>0$ such that \eqref{eq: 3balls} does not hold. Thus, by \Cref{lemma: 3LUR balls}, it is impossible to find balls $B_1$ and $B_2$ of radius one and points $x_i\in \partial B_i$ such that $B_0, B_1, B_2$ don't overlap and ${\rm diam} \{x_0, x_1, x_2\}\leq \eta$. We argue by contradiction and show that the assumption $\gamma(\X) =1$ implies that balls and points as above do exist.

Choose $\e\coloneqq\eta/6$ and let $\delta\in (0,1)$ be such that
\begin{equation}\label{eq: choice delta}
    2\delta< \delta_\X(x_0,\e)\leq \frac\e2.
\end{equation}

If, towards a contradiction, $\gamma(\X)=1$, there exists a set $\P\subseteq \X$ that is $2$-separated and $(1+ \delta)$-dense. Moreover, we can (and do) assume that $0\in \P$. Thus, $\|p\|\geq 2$ for every non-zero $p\in \P$. Let us denote 
\[x_0'\coloneqq (1+\delta)x_0,\qquad x_1'\coloneqq(1+2\delta)x_0; \]
by our assumption that $\P$ is $(1+ \delta)$-dense, there exists $p\in \P\setminus\{0\}$ such that $x_1'\in p+(1+\delta)B_\X$. We now set\footnote{In general, objects decorated with a prime are inflated balls of radius $1+\delta$, or points therein, while objects without the prime are balls of radius one, or their points.} (see \Cref{fig: LUR point})
\[ B_0'\coloneqq (1+\delta)B_\X,\quad B_1\coloneqq p+B_\X, \quad B_1'\coloneqq p+(1+\delta)B_\X. \]
Observe that, since $\|x_1'\|= 1+ 2\delta$ and $\|p\|\geq 2$, we have $1-2\delta\leq \|x_1'-p\|\leq 1+\delta$. Hence, setting $x_1\coloneqq p+\frac{x_1'-p}{\|x_1'-p\|}$, it holds that $x_1\in\partial B_1$ and $\|x_1'-x_1\|\leq 2\delta$. In order to find the point $x_2$ and the ball $B_2$, we shall appeal to the following claim.

\begin{figure}\begin{tikzpicture}[scale=0.4]
    
    \fill[fill=black] (0,0) circle (2pt);
	\node at (-0.3,-0.4) {\tiny{$0$}};

    \draw (0,0) circle (8);
    \shade[color = gray!40, opacity = 0.4] (0,0) circle (8);
    \node at (-3,-3) {\small{$B_0$}};
	\draw[dashed,thick, line width = 0.5 pt, opacity=0.3](0,0)--(0,8);
	\node at (-0.4,4) {\tiny{$1$}};

	\draw[dashed] (0,0) circle (10);
	\shade[color = gray!40, opacity = 0.2] (0,0) circle (10);
    \node at (-6.3,-6.3) {\small{$B_0'$}};		
	\draw[dashed,thick, line width = 0.5 pt, opacity=0.3](0,0)--(-10,0);
	\node at (-5,0.4) {\tiny{$1+\delta$}};

	\fill[fill=black] (8,0) circle (2pt);
    \node at (8.45,0.3) {\tiny{$x_0$}};
    \fill[fill=black] (10,0) circle (2pt);
	\node at (10.45,0.3) {\tiny{$x_0'$}};
	\fill[fill=black] (12,0) circle (2pt);
	\node at (12.45,0.3) {\tiny{$x_1'$}};
    \fill[fill=black] (12.23,-0.45) circle (2pt);
    \node at (12.9,-0.55) {\tiny{$x_1$}};
				
    \draw (16,-7.5) circle (8);
    \shade[color = gray!40, opacity = 0.4] (16,-7.5) circle (8);
    \node at (19,-10.5) {\small{$B_1$}};
    \draw[dashed,thick, line width = 0.5 pt, opacity=0.3](12,0)--(16,-7.5);
    \fill[fill=black] (16,-7.5) circle (2pt);
    \node at (16,-8) {\tiny{$p$}};

    \draw[dashed] (16,-7.5) circle (10);
    \shade[color = gray!40, opacity = 0.2] (16,-7.5) circle (10);
    \node at (22.3,-13.8) {\small{$B_1'$}};

    \draw[dotted] (10,0) circle (4);
	\fill[color = gray!40, opacity = 0.2] (10,0) circle (4);
    \fill[pattern=north east lines, opacity = 0.2] (10,0) circle (4);
    \node at (18,5.3) {\small{$x_0'+2\e B_\X$}};
    \draw [->, thick, dotted, rounded corners] (18,6) to [out=120,in=60] (12,3);

    \fill[fill=black] (10.3,3) circle (2pt);
	\node at (10.75,3.3) {\tiny{$x_2'$}};

\end{tikzpicture}
\caption{Choice of balls $B_1, B_2$ and points $x_1, x_2$ in the proof of\\ \Cref{thm: gamma LUR} (the point $x_1'$ is not necessarily outside of $B_1$).}\label{fig: LUR point}
\end{figure}
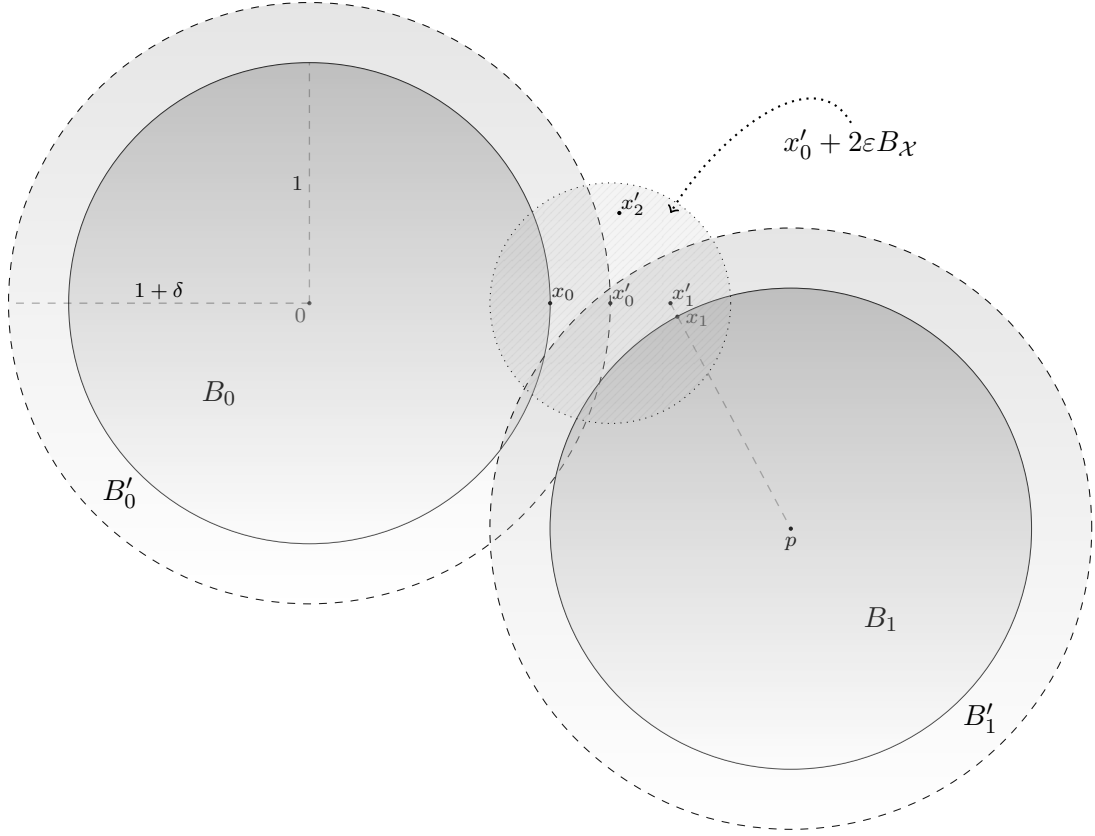

\begin{claim}\label{claim: 3rd close ball} The set $(x_0'+2\e B_\X) \setminus (B_0'\cup B_1')$ is non-empty.
\end{claim}

\begin{proof}[Proof of \Cref{claim: 3rd close ball}] \renewcommand\qedsymbol{$\square$} If $x_0'\notin \inte (B_0'\cup B_1')$, the conclusion of the claim obviously holds. Hence, by \Cref{fact: protected-lur}, we can assume that $x_0'\in \inte B_1'$. Take a point $z'$ in $(\partial B_0')\cap (\partial B_1')$ and consider the open half-line $\ell\coloneqq z'+(0,\infty)(z'-x_0')$. An easy convexity argument implies that:
\begin{itemize}
    \item $\ell$ does not intersect $B_1'$, since $x_0'\in\inte B_1'$;
    \item $\ell$ does not intersect $B_0'$, since $x_0'$ is a LUR point for $B_0'$.
\end{itemize}
Hence, $\ell$ is disjoint from $B_0'\cup B_1'$ and any point in $\ell$ that is close to $z'$ proves the claim, once we show that $\|z'- x_0'\|< 2\e $. For this, since $\frac{x_0'+z'}{2}\in B_1'$, we have
\[ \left\| \frac{x_0'+ z'}{2} \right\|\geq \|p\|- \left\| \frac{x_0'+ z'}{2}- p\right\| \geq1- \delta. \]
Thus, if we put $z\coloneqq \frac{z'}{1+\delta}\in B_{\X}$, we obtain
\[ \left\| \frac{x_0+ z}{2}\right\|= \frac{1}{1+ \delta} \left\| \frac{x_0'+ z'}{2} \right\|\geq \frac{1- \delta}{1+ \delta}. \]
Choosing $r\coloneqq \|x_0-z\|$, the definition of $\delta_\X(x_0,\cdot)$ yields
\[ \delta_\X(x_0,r)\leq1- \left\|\frac{x_0+ z}{2}\right\|\leq \frac{2\delta}{1+ \delta}< \delta_\X(x_0,\e). \]
Since $\delta_{\X}(x_0,\cdot)$ is non-decreasing, we deduce that $r<\e$, thus (recall that $\delta<1$)
\[ \|x_0'-z'\|=r(1+\delta)< \e(1+\delta)< 2\e, \]
and the claim is proved.
\end{proof}

By our previous claim, we can choose $x_2'\in (x_0'+2\e B_\X) \setminus (B_0'\cup B_1')$. Hence, there exists $q\in \P\setminus \{0, p\}$ such that $x_2'\in q+(1+ \delta) B_\X$. We denote $B_2\coloneqq q+ B_\X$. As above, the fact that $\|x_2'\|\leq 1+ \delta+ 2\e$ and $\|q\|\geq 2$ yields $1- \delta- 2\e\leq \|x_2'-q\|\leq 1+ \delta$. Thus, if we denote $x_2\coloneqq q+\frac{x_2'-q}{\|x_2'-q\|}$, we obtain $x_2\in\partial B_2$ and $\|x_2'-x_2\|\leq \delta+2\e$. Finally, observe that:
\begin{itemize}
    \item $\|x_0-x_1\|\leq \|x_0-x_1'\|+ \|x_1'-x_1\|\leq 2\delta+ 2\delta \leq \e$;
    \item $\|x_0-x_2\|\leq \|x_0-x_0'\|+ \|x_0'-x_2'\|+ \|x_2'-x_2\|\leq \delta+ 2\e+ \delta+ 2\e \leq 5\e$;
    \item $\|x_1-x_2\|\leq \|x_1-x_0\|+ \|x_0-x_2\|\leq 6\e$.
\end{itemize}
In particular, we have $\mathrm{diam}\{x_0, x_1, x_2\}\leq \eta$. Thus, we have found balls $B_0, B_1, B_2$ and points $x_0, x_1, x_2$ as in the assumptions of \Cref{lemma: 3LUR balls}, which contradicts our choice of $\eta$ and concludes the proof.
\end{proof}

As a corollary, we obtain that $\gamma(\X) >1$ when $\X$ is uniformly smooth, which compares to the same result for uniformly convex spaces, that we shall prove in \Cref{thm: superreflexive and not}.

\begin{corollary} Let $\X$ be a normed space such that $B_\X$ admits a strongly exposed point $x_0$ in which the norm is Fr\'echet differentiable. Then, $\gamma(\X) >1$. As a consequence, $\gamma(\X) >1$ when $\X$ is a Banach space with the Radon--Nikodym property and its norm is Fr\'echet smooth. In particular, this is the case if $\X$ is a uniformly smooth normed space.
\end{corollary}

It should be pointed out that G\^{a}teaux smoothness of the norm is not sufficient to ensure that $\gamma(\X) >1$. In fact, there exists separable normed spaces $\X$ that are G\^{a}teaux smooth and octahedral, \cite{CH}. For these spaces, $\gamma^*(X) =1$, as we prove in \Cref{thm: octahedral} below.

\begin{proof} It is a consequence of \v{S}mulyan's lemma that, if $x_0\in B_\X$ is strongly exposed and the norm is Fr\'echet differentiable in $x_0$, then $x_0$ is a LUR point, (see, \emph{e.g.}, \cite{AGP}*{Theorem 2.10} for a proof). Therefore, the first part of the result follows directly from \Cref{thm: gamma LUR}. If $\X$ is a Banach space with the Radon--Nikodym property, its unit ball is the closed convex hull of the strongly exposed points; therefore, the first clause applies. Finally, if $\X$ is uniformly smooth, its completion $\hat{\X}$ is also uniformly smooth, hence \emph{a fortiori} Fr\'echet smooth; further, it is super-reflexive, whence it has the Radon--Nikodym property. Thus, the previous clause yields $\gamma(\X)= \gamma(\hat{\X})> 1$.
\end{proof}

We shall now use \Cref{thm: gamma LUR} to give examples of Banach spaces $\X$ such that $\gamma(\X) > 2/K(\X)$, thereby answering Swanepoel's question (\Cref{probl: gamma=2/k}) in the negative.

\begin{example}\label{ex: ell1 +2 R} The Banach space $\ell_1\oplus_2 \R$ clearly admits a LUR point, thus $\gamma(\ell_1\oplus_2\R)>1$ by the above theorem. On the other hand, $K(\ell_1\oplus_2\R)=2$. Thus,
\[ \gamma(\ell_1\oplus_2\R)> \frac{2}{K(\ell_1\oplus_2\R)}. \]

Further, let us observe that $\gamma(\ell_1)=1$, by \cite{Swanepoel} (or \Cref{thm: octahedral} below) and plainly $\ell_1\equiv \ell_1\oplus_1 \R$. It follows that the value of $\gamma(\ell_1\oplus_p \R)$ depends on $p$ (differently from the value of Kottman's constant). We actually do not know what the exact value of $\gamma(\ell_1\oplus_2\R)$ is, see \Cref{probl: ell1 +2 R}.
\end{example}

\begin{remark} More generally, if $\X$ is any Banach space isomorphic to $\ell_1$ (or to $c_0$) and $B_\X$ has a LUR point, then $\gamma(\X)>1$ by the previous theorem, while $K(\X)=2$ by James' distortion theorem. 
\end{remark}

We can also prove that every normed space admits an equivalent norm in which said inequality is strict. Thus, there are reflexive (and even isomorphic to $\ell_2$) counterexamples to Swanepoel's question.
\begin{corollary}\label{cor: renorming counterexample} Every infinite-dimensional normed space $\X$ has an equivalent norm $\n$ such that $B_{(\X,\n)}$ has a LUR point and $K(\X,\n)=2$. Thus, $\gamma(\X,\n)> 2/K(\X,\n)$.
\end{corollary}

\begin{proof} We can write $\X$ as $\Z\oplus \R$, where $\Z$ is a hyperplane in $\X$. By a result of Kottman \cite{Kottman}*{Theorem 7}, there is a norm $\n_\Z$ on $\Z$ such that $K(\Z,\n_\Z)=2$. We can now define an equivalent norm on $\X$ by setting
\[ \|(z,t)\|\coloneqq \sqrt{\|z\|^2_\Z +t^2}, \qquad(z,t)\in \Z\oplus \R, \]
namely, we consider $(\Z,\n_\Z)\oplus_2 \R$. Then, the point $(0,1)$ is a LUR point and $K(\X,\n)=2$, as desired.
\end{proof}

\section{Constructions of discrete subgroups and lattice packings}\label{sec: subgroup}
This section presents two constructions of discrete subgroups in some normed spaces, which directly yield upper bounds for the lattice simultaneous packing and covering constant $\gamma^*(\X)$. The first construction is based on a substantial generalisation of the argument employed in \cite{DRShilbert} to show that $\gamma^*(\X)\leq 2$ for all normed spaces. As a consequence, we obtain information on $\gamma^*(\X)$ when $\X$ is octahedral. The second construction employs the standard lattice tiling of $c_0$ and generalises the construction to a vast class of $\C(\K)$ spaces. Our first theorem is the main result of the section and contains the first construction mentioned above.

\begin{theorem}\label{thm: subgroup theta} Let $\X$ be an infinite-dimensional normed space of density $\kappa$. Suppose that there exists $\theta>1$ with the following property: for every subspace $\Z$ of $\X$ with $\dens(\Z)< \kappa$, there exists a vector $x\in S_\X$ such that
\begin{equation} \label{eq: x far from S_Z}
    \dist(x, S_\Z)\geq \theta.
\end{equation}
Then,
\[ \gamma^*(\X) \leq \frac{2}{\theta}. \]
\end{theorem}

It will be apparent from the proof below that the result also holds when $\theta=1$. However, in this case, we don't obtain any improvement over the bound $\gamma^*(\X) \leq2$, which is already known from \cite{DRShilbert}. Therefore, we only consider the case that $\theta>1$.

\begin{proof} To begin with, we observe that \eqref{eq: x far from S_Z} implies that
\begin{equation} \label{eq: x far from Z}
    \dist(x, \Z)\geq \frac{\theta}{2}.
\end{equation}
In fact, for every $z\in S_\Z$, consider the function $f(t)\coloneqq \|x-tz\|$ ($t\in\R)$. We will show that $f(t)\geq \theta/2$ for $t\geq 0$, which implies \eqref{eq: x far from Z}. Plainly, $f$ is convex and $1$-Lipschitz; further, by assumption, $f(0)=1$ and $f(1)\geq \theta$. Since $f$ is convex and $\theta>1$, $f(t)\geq \theta$ when $t\geq 1$. Moreover, the fact that $f$ is $1$-Lipschitz yields that
\[ f(t)\geq \max\{1-t, t+\theta -1\} \qquad\mbox{for }t\in [0,1]. \]
Computing the minimum of the function $t\mapsto \max\{1-t, t+\theta -1\}$ gives $f(t)\geq \theta/2$ for $t\in [0,1]$, as required.

We now enter the core part of the proof and we construct the subgroup witnessing that $\gamma^*(\X) \leq \frac{2}{\theta}$. Choose a dense set $\{u_\alpha\}_{\alpha< \kappa}$ in $\X$ such that $u_0=0$. We shall construct, by transfinite induction, an increasing chain $(\D_\alpha) _{\alpha< \kappa}$ of subgroups of $\X$ such that, for all $\alpha<\kappa$, the following properties hold:

\begin{enumerate}
    \item\label{i: gen} $\D_\alpha$ is generated by at most $|\alpha|$ elements;
    \item\label{ii: septd} $\D_\alpha$ is $\theta$-separated;
    \item\label{iii: close} there exists $d\in \D_\alpha$ such that $\|d- u_\alpha\|\leq 1$.
\end{enumerate}

Once we have such subgroups at our disposal, we set $\D\coloneqq \bigcup_{\alpha< \kappa} \D_\alpha$. Because the chain $(\D_\alpha) _{\alpha< \kappa}$ is increasing, it is clear that $\D$ is a subgroup of $\X$, which is $\theta$-separated. Further, for every $\e>0$, \eqref{iii: close} and the fact that $\{u_\alpha\}_{\alpha< \kappa}$ is dense in $\X$ imply that $\D$ is $(1+ \e)$-dense in $\X$. By scaling, $\frac{1}{1+ \e}\D$ is $1$-dense and $\frac{\theta}{1+\e}$-separated, which, by \eqref{eq: 2/gamma(X)}, yields
\[ \frac{2}{\gamma^*(\X)}\geq \frac{\theta}{1+\e}. \]
As $\e>0$ is arbitrary, the conclusion follows.

Therefore, we only have to construct a chain $(\D_\alpha)_{\alpha< \kappa}$ as above, which we shall do by transfinite induction. To begin with, we set $\D_0\coloneqq \{0\}$, so that \eqref{iii: close} holds because $u_0=0$, while \eqref{i: gen} and \eqref{ii: septd} are obvious. Assume now, by transfinite induction, to have already constructed subgroups $(\D_\alpha)_{\alpha<\beta}$ with the above properties, for some $\beta< \kappa$. Consider the set $\bigcup_{\alpha< \beta} \D_\alpha$; arguing as before, it is clear that $\bigcup_{\alpha< \beta} \D_\alpha$ is a subgroup of $\X$, which is generated by at most $|\beta|$ elements and is $\theta$-separated. If there exists $d\in \bigcup_{\alpha< \beta} \D_\alpha$ such that $\|d- u_\beta\|\leq 1$, the induction step is concluded by just setting $\D_\beta\coloneqq \bigcup_{\alpha< \beta} \D_\alpha$. Hence, using that $\bigcup_{\alpha< \beta} \D_\alpha=- \bigcup_{\alpha< \beta} \D_\alpha$ we can assume without loss of generality that 
\begin{equation}\label{eq: norm >1}
    \|d+ u_\beta\|> 1 \qquad\mbox{for all } d\in \bigcup_{\alpha< \beta} \D_\alpha.
\end{equation}

Consider the subspace $\Z$ of $\X$ defined by
\[ \Z\coloneqq \overline{\spn}\left\{\bigcup_{\alpha< \beta} \D_\alpha, u_\beta \right\}. \]
In the light of \eqref{i: gen}, we see that $\dens(\Z)\leq |\beta|< \kappa$, therefore our assumption implies the existence of a unit vector $x_\beta\in S_\X$ such that $\dist(x_\beta, S_\Z)\geq \theta$. By the argument at the beginning of the proof, \eqref{eq: x far from Z} holds as well. We now define the subgroup
\[ \D_\beta\coloneqq \left(\bigcup_{\alpha< \beta} \D_\alpha\right)\oplus (u_\beta- x_\beta)\ZZ. \]
The validity of \eqref{iii: close} is clear, because the vector $u_\beta- x_\beta\in \D_\beta$ has distance $1$ from $u_\beta$, and \eqref{i: gen} is also immediate (when $\beta$ is a finite ordinal, $\bigcup_{\alpha< \beta} \D_\alpha= \D_{\beta-1}$ is generated by at most $\beta-1$ elements, so \eqref{i: gen} also holds in this case). Thus, we only have to check that $\D_\beta$ is $\theta$-separated, equivalently (as $\D_\beta$ is a subgroup), that every non-zero element of $\D_\beta$ has norm at least $\theta$.

Take any non-zero element $d+k(u_\beta- x_\beta)\in \D_\beta$, where $d\in \bigcup_{\alpha<\beta} \D_\alpha$ and $k\in \ZZ$. If $k=0$, then $\|d\|\geq \theta$ by \eqref{ii: septd} of the inductive assumption; hence, we assume that $k$ is non-zero. Further, as $d\in \bigcup_{\alpha<\beta} \D_\alpha$ if and only if $-d\in \bigcup_{\alpha<\beta} \D_\alpha$, we can assume that $k>0$. We distinguish two cases depending on whether $k=1$ or $k\geq 2$.
\begin{itemize}
    \item If $k=1$, we use \eqref{eq: x far from S_Z} to show that $\|d+(u_\beta- x_\beta)\|= \|x_\beta- (d+ u_\beta)\|\geq \theta$. Consider the function
    \[ f(t)\coloneqq \left\| x_\beta - t\frac{d+ u_\beta}{\|d+ u_\beta\|} \right\| \qquad t\in\R. \]
    Then, $f(0)=1$ and, by \eqref{eq: x far from S_Z}, $f(1)\geq \theta>1$, since $d+ u_\beta\in \Z$. The convexity of $f$ and the fact that $\|d+ u_\beta\|> 1$ by \eqref{eq: norm >1} then yield
    \[ f(\|d+ u_\beta\|)= \|x_\beta- (d+ u_\beta)\|\geq \theta. \]
    \item Instead, if $k\geq 2$ we use \eqref{eq: x far from Z}. Notice that
    \[ \|d+k(u_\beta- x_\beta)\|= k \left\| \left(\textstyle{\frac{d}{k}}+ u_\beta\right)- x_\beta \right\|\geq 2\left\| \left(\textstyle{\frac{d}{k}}+ u_\beta\right)- x_\beta \right\|. \]
    Since the vector $\frac{d}{k}+ u_\beta$ belongs to $\Z$, by \eqref{eq: x far from Z} its distance from $x_\beta$ is at least $\frac{\theta}{2}$. Thus, $\|d+k(u_\beta- x_\beta)\|\geq \theta$ in this case as well.
\end{itemize}
Hence, $\D_\beta$ is $\theta$-separated, which concludes the induction step and the proof.
\end{proof}

\subsection{Octahedral and \texorpdfstring{$\C(\K)$}{C(K)} spaces}\label{sec: gamma=1}
We now move to the second part of the section, where we apply the above general result and determine two classes of normed spaces having the best possible value for $\gamma^*$, namely $\gamma^*(\X)= 1$ (which obviously implies that $\gamma(\X)=1$ as well). The first class is that of octahedral normed spaces and the corresponding result is just a direct consequence of \Cref{thm: subgroup theta}. For the second class, that of $\C(\K)$ spaces for zero-dimensional $\K$, we generalise the construction of the tiling of $c_0$ via the even integers grid (see also \Cref{prop: L_infty}).

We begin by recalling the definition of octahedral normed space. A normed space $\X$ is \emph{octahedral} if, for every finite-dimensional subspace $\Z$ of $\X$ and every $\e>0$, there is $x\in S_\X$ such that
\[ \|z+\lambda x\|\geq (1-\e)(\|z\| + |\lambda|) \]
for every $z\in \Z$ and every $\lambda\in \R$.

Recently, in \cite{CLL}*{Definition 5.3} the definition of octahedral normed space was generalised to encompass infinite-dimensional subspaces as well. For an infinite cardinal $\kappa$, a normed space $\X$ is \emph{$(<\kappa)$-octahedral} if for every subspace $\Z$ of $\X$ with $\dens(\Z)< \kappa$ and every $\e>0$ , there is $x\in S_\X$ such that
\[ \|z+\lambda x\|\geq (1-\e)(\|z\| + |\lambda|) \]
for every $z\in \Z$ and every $\lambda\in \R$. Clearly, $(<\omega)$-octahedrality merely reduces to octahedrality (recall that, by our convention, $\dens(\Z)< \omega$ means that $\Z$ is finite dimensional).

It is clear that every $(<\kappa)$-octahedral normed space $\X$ of density $\kappa$ satisfies the assumptions of \Cref{thm: subgroup theta} with $\theta= 2-\e$, for every $\e>0$. Therefore, the following theorem is an immediate consequence of \Cref{thm: subgroup theta}.

\begin{theorem}\label{thm: octahedral} Every $(<\kappa)$-octahedral normed space $\X$ of density $\kappa$ satisfies $\gamma^*(\X)=1$. In particular, every separable octahedral normed space $\X$ satisfies $\gamma^*(\X)=1$.
\end{theorem}

As we will see later, the separability of $\X$ can not be dispensed with, in the second clause of the previous theorem. In fact, we will see in \Cref{ex: octa gamma>1} that there exist octahedral Banach spaces $\X$ of density $\omega_1$ and such that $\gamma(\X)$ is as close to $2$ as we wish. In \Cref{sec: gamma=2} we will even give an example of an octahedral Banach space $\X$ (having density $\omega_\omega$) such that $\gamma(\X)=2$.

\begin{remark}\label{rmk: <kappa octa examples} Beside the examples and classes of octahedral Banach spaces that we mentioned in the Introduction, the previous theorem allows us to obtain more examples of Banach spaces $\X$ such that $\gamma^*(\X)=1$, by gleaning examples of $(<\kappa)$-octahedral Banach spaces from \cites{ACLLR, AMR_ell1(kappa), CLL}. Further, it can be directly checked that Banach spaces $\C(\{0,1\}^\kappa)$ are $(<\kappa)$-octahedral and that $\X\oplus_1 \Y$ is $(<\kappa)$-octahedral whenever $\X$ is $(<\kappa)$-octahedral ($\Y$ being arbitrary). Moreover, for every measure space $(\M, \Sigma, \mu)$, the Banach space $L_1(\mu)$ is $(<\kappa)$-octahedral, where $\kappa$ is the density of $L_1(\mu)$. Since, to the best of our knowledge, this perhaps folklore result has not appeared in the literature, we will prove it (in a more general form) in \Cref{prop: L_p p-octa}. Finally, in order to connect with the second part of this section, we also recall here that $\C(\K)$ is octahedral if (and only if) $\K$ is perfect (see the references in the proof of \Cref{lem: octa iff K perfect}); therefore, $\gamma^*(\C(\K))=1$ when $\K$ is a perfect compact metric space.
\end{remark}

We now move to the class of $\C(\K)$ spaces. Recall that a compact topological space is \emph{zero-dimensional} if it admits a basis consisting of clopen sets (equivalently, it is \emph{totally disconnected}, namely all its connected components are singletons). Further, a compact topological space is \emph{extremally disconnected} if the closure of every open subset is open (and hence clopen). Equivalently, disjoint open sets have disjoint closures.

\begin{theorem}\label{thm: C(K)} If $\K$ is zero-dimensional, then $\gamma^*(\C(\K))=1$. Further, if $\K$ is extremally disconnected, $\C(\K)$ admits a lattice tiling by balls.
\end{theorem}

Recall that by the Goodner--Kelley--Nachbin theorem (see, \emph{e.g.}, \cite{AK}*{Section 4.3}) a Banach space is $1$-injective if and only if it is isometrically isomorphic to $\C(\K)$, for some extremally disconnected $\K$; therefore, the above theorem implies in particular the appealing fact that $1$-injective Banach spaces admit a lattice tiling by balls.

\begin{proof} In both clauses of the proof we consider the subgroup $\C(\K;2\ZZ)$ of $\C(\K)$ comprising all continuous functions on $\K$ with values in the even integers $2\ZZ$. Plainly, $\C(\K;2\ZZ)$ is $2$-separated. In order to clinch the proof it is enough to prove the following two claims: if $\K$ is zero-dimensional, $\C(\K;2\ZZ)$ is $(1+\e)$-dense for every $\e>0$, and if $\K$ is extremally disconnected, then $\C(\K;2\ZZ)$ is $1$-dense.

For the first case, suppose that $\K$ is zero-dimensional and fix any function $f\in \C(\K)$. Because $\K$ is zero-dimensional, there are a partition $\{U_1,\dots,U_n\}$ of $\K$ into clopen sets and real numbers $\lambda_1\dots, \lambda_n$ such that 
\[ |f(x) - \lambda_j|< \e \;\mbox{ for all }\; x\in U_j \]
(take a clopen cover such that $f$ has small oscillation on each clopen set in the cover; pass to a finite subcover and use the fact that clopen sets constitute a Boolean algebra to obtain a partition). For each $j=1,\dots,n$ there is an even integer $k_j\in 2\ZZ$ such that $|\lambda_j- k_j|\leq 1$. Then, the function
\[ g\coloneqq \sum_{j=1}^n k_j\cdot \bone_{U_j} \]
belongs to $\C(\K;2\ZZ)$ and $\|f-g\|\leq 1+\e$, as desired.

Next, we assume that $\K$ is extremally disconnected and fix $f\in \C(\K)$. The sets 
\[ V_k\coloneqq \{x\in \K\colon 2k-1< f(x)< 2k+1\} \qquad (k\in\ZZ) \]
are open and disjoint, hence their closures $\overline{V_k}$ are disjoint clopen sets. Further, on $\overline{V_k}$, $f$ attains values in the interval $[2k-1, 2k+1]$. Notice that, as $f$ is bounded, only finitely many $\overline{V_k}$ are non-empty. Hence, the set $\K_0\coloneqq \K\setminus \bigcup_{k\in\ZZ} \overline{V_k}$ is clopen; further, we have
\[ \K_0= \bigcup_{k\in\ZZ} \{x\in \K_0\colon f(x)=2k+1\} \]
(as above, notice that this is actually a finite union). Thus, the clopen set $\K_0$ is expressed as a disjoint union of finitely many closed sets; therefore, all these sets are in fact clopen. We may thus consider the clopen sets
\[ U_k\coloneqq \overline{V_k}\cup \{x\in \K_0\colon f(x)= 2k+1\}; \]
as before, $f$ has values in the interval $[2k-1, 2k+1]$ on the set $U_k$. Further, $\{U_k\}_{k\in\ZZ}$ is a partition of $\K$ in finitely many clopen sets. Thus, the function
\[ g\coloneqq \sum_{k\in\ZZ} 2k\cdot \bone_{U_k} \]
belongs to $\C(\K;2\ZZ)$ and clearly $\|f-g\|\leq 1$, thereby concluding the proof.
\end{proof}

\begin{remark} An alternative, slightly shorter but not self-contained, proof of the second claim could be obtained by using the fact that $\C(\K)$ is \emph{order-complete} (namely, every subset of $\C(\K)$ with an upper bound admits a least upper bound) if and only if $\K$ is extremally disconnected. We just sketch the argument here. For $f\in \C(\K)$, let $h_0$ be the least upper bound of the set
\[ \mathcal{L}\coloneqq \{h\in \C(\K;\ZZ\setminus 2\ZZ)\colon h\leq f \}. \]
It is not hard to prove that $h_0(x)\leq f(x)\leq h_0(x)+2$ for all $x\in \K$ (in fact, otherwise there would be an integer $k\in \ZZ$ and a clopen subset $U$ of $\K$ such that $f(x)> 2k+1> h_0(x)$ for all $x\in U$, which readily contradicts the fact that $h_0$ is an upper bound). Then, the function $g\coloneqq h_0 + \bone_\K\in \C(\K;2\ZZ)$ satisfies $g-\bone_\K\leq f\leq g+ \bone_\K$, whence $\|f-g\|\leq 1$.
\end{remark}

\begin{remark} The above construction clearly generalises the construction of the tiling in $c_0$ or $\ell_\infty$ by balls of radius $1$ and centers the even integers grid. It is of course natural to wonder whether the packing obtained in the above proof is a tiling also when $\K$ is only assumed to be zero-dimensional. As it turns out, this is not the case, because $\C(\K;2\ZZ)$ might fail to be $1$-dense. We illustrate this in the Banach space $c$ of convergent sequences. Consider the sequence $x= (x_n)_{n=1}^\infty$ given by $x_n= 1+ (-1)^n/n$. Suppose that there is a $2\ZZ$-valued sequence $y= (y_n)_{n=1}^\infty \in c$ such that $\|x-y\|\leq 1$. Then $y_{2n}=2$, while $y_{2n+1}=0$, contradicting the fact that $(y_n)_{n=1}^\infty$ must admit a limit. Notice that the same construction actually works for every $\K$ that contains non-trivial convergent sequences.
\end{remark}

\section{\texorpdfstring{$\phi$}{phi}-octahedral normed spaces and lower bounds on \texorpdfstring{$\gamma(\X)$}{gamma(X)}} \label{sec: phi-octahedral}
In this section we introduce the notion of $\phi$-octahedral normed space, by replacing the linear growth in the `orthogonal' direction with a growth lower bounded by a fixed modulus $\phi$. We study the stability of $\phi$-octahedrality under $\ell_p$-sums and we then give a lower bound for $\gamma(\X)$ when $\X$ is $\phi$-octahedral (\Cref{thm: gamma of phi-octa}). As an application of this lower bound, in \Cref{sec: gamma=2} we also give examples of Banach spaces $\X$ such that $\gamma(\X)=2$.

\begin{definition}\label{def: modulus} We say that  $\phi\colon [0,\infty)\to [0,\infty)$ is a \emph{modulus} if 
\begin{enumerate}[label=($\phi$\arabic*), ref=$\phi$\arabic*]
    \item \label{phi_i: cont 0} $\phi(0) =0$ and $\phi$ is continuous in $0$, and
    \item \label{phi_i: phi/t} the function $t\mapsto \frac{\phi(t)}{t}$ is  non-decreasing on $(0,\infty)$.
\end{enumerate}
The modulus $\phi$ is a \emph{positive modulus} if $\phi(t)>0$ when $t>0$.
\end{definition}

As an example, every convex function $\phi\colon [0,\infty)\to [0,\infty)$ that vanishes at $0$ is a modulus (but the converse is not true). Further, note that \eqref{phi_i: phi/t} implies in particular that $\phi$ is non-decreasing.  We begin with a reformulation of this condition.

\begin{lemma}\label{lemma: basic modulus} Condition \eqref{phi_i: phi/t} is equivalent to
\begin{enumerate}[label=\rm{($\phi$\arabic*)}, ref=$\phi$\arabic*] \setcounter{enumi}{2}
\item \label{phi_i: modulus iff} $\phi(\lambda t)\leq \lambda \phi(t)$ for all $\lambda\in (0,1)$ and $t\in (0,\infty)$.
\end{enumerate}
As a consequence, $\psi\circ \phi$ is a (positive) modulus whenever $\psi$ and $\phi$ are.
\end{lemma}

\begin{proof} If \eqref{phi_i: phi/t} holds, the fact that $\lambda t< t$ implies that $\frac{\phi(\lambda t)}{\lambda t}\leq \frac{\phi(t)}{t}$, whence $\phi(\lambda t)\leq \lambda \phi(t)$. Conversely, fix $s,t\in (0,\infty)$ with $s<t$ and set $\lambda\coloneqq s/t\in (0,1)$. Then $\phi(s)= \phi(\lambda t)\leq \frac{s}{t} \phi(t)$, and $\frac{\phi(s)}{s}\leq \frac{\phi(t)}{t}$.

For the second part, note that, for $\lambda\in (0,1)$, $\psi(\phi(\lambda t))\leq \psi(\lambda \phi(t))\leq \lambda \psi( \phi(t))$, where we used the fact that $\psi$ is non-decreasing. The other assertions are clear.
\end{proof}

The most important examples of moduli for us will be the moduli $\p_\X$ (\Cref{def: phi_X}) and $\phi_p$ defined by $\phi_p(t)\coloneqq (1+ t^p)^{1/p}-1$ related to the $\ell_p$ spaces (\Cref{ex: phi_p}). The fact that $\phi_p$ is indeed a positive modulus is an easy computation that we omit.

\begin{definition}\label{def: phi-octa} Let $\kappa$ be an infinite cardinal and $\phi$ a positive modulus. A normed space $\X$ is \emph{$(<\kappa)$-$\phi$-octahedral} if, for every $\e>0$ and every subspace $\Z$ of $\X$ with $\dens(\Z)< \kappa$, there exists a vector $x\in S_\X$ such that
\begin{equation} \label{eq: phi-orthogonal}
	\|z+ \lambda x\|\geq (1-\e)\big(1+ \phi(|\lambda|) \big) \qquad \mbox{for all }\, z\in S_\Z \mbox{ and } \lambda\in \R.
\end{equation}
When $\kappa=\omega$, we just say that $\X$ is \emph{$\phi$-octahedral}. We will also say that a vector $x$ satisfying \eqref{eq: phi-orthogonal} is \emph{$(\phi,\e)$-orthogonal} to $\Z$ .
\end{definition}

Before moving to our main results, we collect some basic remarks and examples.
\begin{remark} The notion of $(<\kappa)$-$\phi$-octahedrality is much more general than mere $(<\kappa)$-octahedrality. In fact:
\begin{enumerate}
    \item $(<\kappa)$-$\phi$-octahedrality is a generalisation of $(<\kappa)$-octahedrality. Not only the definition reduces to the notion of $(<\kappa)$-octahedrality when $\phi$ is the identity, but also every $(<\kappa)$-octahedral normed space is $(<\kappa)$-$\phi$-octahedral for all moduli $\phi$ such that $\phi(t)\leq t$ for all $t\in [0,\infty)$. (Notice that the existence of some $\phi$-octahedral normed space implies that $\phi(t)\leq t$ for all $t\in [0,\infty)$.)
    \item On the other hand, all uniformly convex normed spaces of density $\kappa$ are $(<\kappa)$-$\phi$-octahedral (\Cref{thm: UC implies phi-octa}). Therefore, $\phi$-octahedral spaces can be quite far from octahedral ones\footnote{And, as far as this manuscript is a preprint, we explicitly welcome criticisms or alternative terminology.}. Further, $\ell_1$ is, say, $\phi_2$-octahedral, whence $\phi_p$-octahedral Banach spaces need not contain a copy of $\ell_p$, nor be reflexive. On the other hand, it is not inconceivable that every Banach space containing $\ell_p(\kappa)$ might admit a $(<\kappa)$-$\phi_p$-octahedral norm (which would generalise a result from \cites{CLL, AMR_ell1(kappa)}).
\end{enumerate}
\end{remark}

Next, we rephrase the notion of $\phi_p$-octahedrality and show that $\ell_p(\kappa)$ is $(<\kappa)$-$\phi_p$-octahedral. The more general fact that $L_p(\mu)$ is $(<\kappa)$-$\phi_p$-octahedral, where $\kappa$ is the density of $L_p(\mu)$, will be proved in \Cref{prop: L_p p-octa}.

\begin{example}\label{ex: phi_p} In the case of $\phi_p$, the inequality in \eqref{eq: phi-orthogonal} rewrites as
\[ \|z+ \lambda x\|\geq (1-\e)\big( 1+ |\lambda|^p \big)^{1/p} \qquad \mbox{for all }\, z\in S_\Z \mbox{ and } \lambda\in \R. \]
By homogeneity, this is equivalent to 
\begin{equation}\label{eq: p-octahedral}
    \|z+ \lambda x\|\geq (1-\e)\big( \|z\|^p+ |\lambda|^p \big)^{1/p} \qquad \mbox{for all }\, z\in \Z \mbox{ and } \lambda\in \R.
\end{equation}
Thus, $\X$ is $(<\kappa)$-$\phi_p$-octahedral if and only if for every $\e>0$ and every subspace $\Z$ of $\X$ with $\dens(\Z)< \kappa$, there exists $x\in S_\X$ such that \eqref{eq: p-octahedral} holds. Notice further that, again by homogeneity, in \eqref{eq: p-octahedral} it is enough to assume that $\lambda= 1$.
\end{example}

As a consequence of this, we observe that $\ell_p(\kappa)$ is $(<\kappa)$-$\phi_p$-octahedral (the proof is identical to the case $p=1$, and we just give it for the sake of completeness).

\begin{lemma}\label{lemma: ell_p octa} Let $\kappa$ be an infinite cardinal, $(\X_\alpha)_{\alpha< \kappa}$ be a family of non-zero normed spaces, and $p\in [1,\infty)$. Then $(\bigoplus_{\alpha< \kappa} \X_\alpha)_{\ell_p}$ is $(<\kappa)$-$\phi_p$-octahedral. In particular, $\ell_p(\kappa)$ is $(<\kappa)$-$\phi_p$-octahedral.
\end{lemma}

\begin{proof} Fix $p\in [1,\infty)$, an infinite cardinal $\kappa$, non-zero normed spaces $(\X_\alpha)_{\alpha< \kappa}$, and let $\Z$ be a subspace of $(\bigoplus_{\alpha< \kappa} \X_\alpha)_{\ell_p}$ with $\dens(\Z)< \kappa$. We distinguish two cases depending on whether $\kappa= \omega$ or $\kappa\geq \omega_1$.

If $\kappa$ is uncountable, because every vector of $(\bigoplus_{\alpha< \kappa} \X_\alpha)_{\ell_p}$ is countably supported and $\dens(\Z)< \kappa$, there exists $\alpha\in \kappa$ such that $z(\alpha)=0$ for each $z\in \Z$. Therefore, if we pick any unit vector $e_\alpha\in \X_\alpha$, we have
\[ \|z+ e_\alpha\|^p= \|z\|^p+ 1,\; \mbox{ for all } z\in \Z.\]

Instead, if $\kappa$ is countable, we fix $\e>0$ (and notice that $\Z$ is finite dimensional in this case). Since $\lim_{t\to \infty}\frac{t-1}{(t^p+1)^{1/p}} =1$, there exists $M\in \R$ such that $t-1\geq (1-\e)(t^p+1)^{1/p}$ for every $t\geq M$. Thus, for all $z\in \Z$ with $\|z\|\geq M$ and all $x\in (\bigoplus_{k< \omega} \X_k)_{\ell_p}$ with $\|x\| =1$,
\begin{equation}\label{eq: p-octa for ||z|| large}
    \|z+ x\|\geq \|z\|-1\geq (1-\e)\big(\|z\|^p+ 1 \big)^{1/p}.
\end{equation}
Therefore, it is enough to find $x\in (\bigoplus_{k< \omega} \X_k)_{\ell_p}$ with $\|x\|=1$ that satisfies \eqref{eq: p-octahedral} for $\lambda=1$ and all $z\in \Z$ with $\|z\|\leq M$. For each $k\in \N$, take a unit vector $e_k\in \X_k$. For a fixed $z\in \Z$ and all sufficiently large $n\in \N$
\[ \|z+ e_n\|^p= \|z\|^p- \|z(n)\|^p+ \|z(n)+ e_n\|^p\geq (1-\e)^p (\|z\|^p+ 1). \]
However, the set $\{z\in \Z\colon \|z\|\leq M\}$ is compact because $\Z$ is finite dimensional. Thus, there is $n\in \N$ so that the above inequality (possibly, with $\e$ replaced by $2\e$) holds for all $z\in \Z$ with $\|z\|\leq M$. This shows that $(\bigoplus_{k< \omega} \X_k)_{\ell_p}$ is $\phi_p$-octahedral.
\end{proof}

Before we proceed, it is our duty to give at least one example of a normed space that is not $\phi$-octahedral for any (positive) modulus $\phi$, the simplest being the space $c_0$. The same argument works for $\C_0(\K)$, for every locally compact space $\K$ and actually shows that $\phi$-octahedrality is equivalent to octahedrality for $\C_0(\K)$ spaces.

\begin{lemma}\label{lem: octa iff K perfect} For a locally compact space $\K$ the following are equivalent:
\begin{enumerate}
    \item \label{i: K perfect} $\K$ is perfect;
    \item \label{i: C(K) Daug} $\C_0(\K)$ has the Daugavet property;
    \item \label{i: C(K) octa} $\C_0(\K)$ is octahedral;
    \item \label{i: C(K) phi-octa} $\C_0(\K)$ is $\phi$-octahedral, for some positive modulus $\phi$.
\end{enumerate}
\end{lemma}
The equivalence between \eqref{i: K perfect}, \eqref{i: C(K) Daug}, and \eqref{i: C(K) octa} is well-known and we just repeat if for clarity.

\begin{proof} The implication \eqref{i: K perfect}$\implies$\eqref{i: C(K) Daug} is well-known, see, \emph{e.g.}, \cite{KMRW_Daug_book}*{Theorem 3.3.1}, while \eqref{i: C(K) Daug}$\implies$\eqref{i: C(K) octa} holds for all Banach spaces, and \eqref{i: C(K) octa}$\implies$\eqref{i: C(K) phi-octa} is obvious. Thus, we only need to prove that $\C_0(\K)$ is not $\phi$-octahedral, for any positive modulus $\phi$, provided that $\K$ has an isolated point (the proof of this implication is essentially the same as the proof that \eqref{i: C(K) Daug}$\implies$\eqref{i: K perfect}). Let $x_0\in \K$ be an isolated point and $\Z\coloneqq \spn\{\bone_{\{x_0\}} \}$. If $\C_0(\K)$ were $\phi$-octahedral, for each $\e>0$, there would be a unit vector $f\in \C_0(\K)$ such that $\|\eta \bone_{\{x_0\}}+ f\|\geq (1-\e)(1+ \phi(1))$, for $\eta= \pm1$. Since $|f(x_0)|\leq 1$, we can choose $\eta= \pm1$ so that $|\eta+ f(x_0)|\leq 1$. Thus, $\|\eta \bone_{\{x_0\}}+ f\|\leq 1$. This yields $(1-\e)(1+ \phi(1))\leq 1$, which, for sufficiently small $\e>0$, is a contradiction.
\end{proof}

We now move to results concerning stability of $(<\kappa)$-$\phi$-octahedrality under direct sums, the first pertaining to general moduli and the second being an improvement in the case of the moduli $\phi_p$.

\begin{theorem}\label{thm: phi-octa sum Y} Let $\X$ be a $(<\kappa)$-$\phi$-octahedral normed space and $p\in[1,\infty)$. Then $\X\oplus_p \Y$ is $(<\kappa)$-$(\phi_p\circ \phi)$-octahedral for every normed space $\Y$.
\end{theorem}

\begin{proof} Let us denote by $P\colon \X\oplus_p \Y\to \X$ the canonical projection onto the first component. Fix $\e>0$ and let $\Z$ be an arbitrary subspace of $\X\oplus_p \Y$ such that $\dens(\Z)< \kappa$. Then, $P(\Z)$ is a subspace of $\X$ and $\dens(P(\Z))< \kappa$ as well. Therefore, the definition of $(<\kappa)$-$\phi$-octahedrality of $\X$ provides us with a vector $x\in S_\X$ with the property that $x$ is $(\phi,\e)$-orthogonal to $P(\Z)$ . We will show that $(x,0)$ is $(\phi_p\circ \phi,\e)$-orthogonal to $\Z$, which concludes the proof. Thus, let us fix $z=(z_1,z_2)\in S_\Z$ and $\lambda\in\R$. To begin with, since $\|z_1\|\leq 1$ and $\phi$ is a modulus, we obtain
\begin{eqnarray*}
   \|z_1+ \lambda x\|&=& \|z_1\| \left\|\frac{z_1}{\|z_1\|}+ \frac{\lambda}{\|z_1\|} x \right\| \\
    &\geq& (1-\e)\left(\|z_1\|+ \|z_1\|\phi\left( \frac{|\lambda|}{\|z_1\|} \right) \right)\geq (1-\e) \big(\|z_1\|+ \phi(|\lambda|)\big),
\end{eqnarray*}
where the last inequality follows from \eqref{phi_i: modulus iff}. Combining this inequality with the standard fact that $(a+b)^p\geq a^p+ b^p$ for $a,b\geq 0$, we get
\begin{eqnarray*}
 	\|(z_1+ \lambda x,z_2)\|&=& \big(\|z_1+ \lambda x\|^p+ \|z_2\|^p \big)^{1/p}\\
	&\geq& (1-\e)\Big(\big(\|z_1\|+ \phi(|\lambda|)\big)^p+ \|z_2\|^p \Big)^{1/p}\\
    &\geq& (1-\e)\big(\|z_1\|^p+ \phi(|\lambda|)^p+ \|z_2\|^p \big)^{1/p}\\
    &=& (1-\e)\big(1+ \phi(|\lambda|)^p\big)^{1/p}\\
    &=& (1-\e)\big(1+ \phi_p\circ\phi (|\lambda|) \big).
\end{eqnarray*}
This means that $(x,0)$ is $(\phi_p\circ \phi,\e)$-orthogonal to $\Z$, as desired.
\end{proof}

In the particular case when the modulus has the form $\phi_p$, we can obtain the following stronger result, which illustrates the importance of the moduli $\phi_p$. Further, it generalises the fact that $\X\oplus_1 \Y$ is $(<\kappa)$-octahedral when $\X$ is $(<\kappa)$-octahedral.

\begin{proposition}\label{prop: p-octa sum Y} Let $\X$ be a $(<\kappa)$-$\phi_p$-octahedral normed space and $1\leq r \leq p$. Then $\X\oplus_r \Y$ is $(<\kappa)$-$\phi_p$-octahedral for every normed space $\Y$.
\end{proposition}

Notice that the proposition also implies that $\X\oplus_r \Y$ is $(<\kappa)$-$\phi_r$-octahedral, when $r>p$. In fact, if $r>p$, $\X$ is also $(<\kappa)$-$\phi_r$-octahedral, whence $\X\oplus_r \Y$ is $(<\kappa)$-$\phi_r$-octahedral.

\begin{proof} The desired vector is chosen in the same way as in the previous proof, while the computation showing the $(\phi_p,\e)$-orthogonality differs. Fix $\e>0$ and let $\Z$ be an arbitrary subspace of $\X\oplus_r \Y$ such that $\dens(\Z)< \kappa$. Letting $P\colon \X\oplus_r \Y\to \X$ be the canonical projection, $P(\Z)$ is a subspace of $\X$ and $\dens(P(\Z))< \kappa$. Thus, the $(<\kappa)$-$\phi_p$-octahedrality of $\X$ (\Cref{ex: phi_p}) provides us with a vector $x\in S_\X$ with the property that
\[ \|z_1+ x\|\geq (1-\e)\big(\|z_1\|^p+ 1\big)^{1/p} \qquad \mbox{for all }\, z_1\in P(\Z). \]

We claim that the vector $(x,0)\in \X\oplus_r \Y$ witnesses that $\X\oplus_r \Y$ is $(<\kappa)$-$\phi_p$-octahedral. Therefore, we fix any element $z=(z_1,z_2)\in \Z$. The inequality we aim to prove is
\begin{eqnarray*}
    \|(z_1+ x,z_2)\|&\geq& (1-\e)\Big(\|(z_1,z_2)\|^p+ 1\Big)^{1/p}\\
    &=& (1-\e)\Big(\big(\|z_1\|^r+ \|z_2\|^r\big)^{p/r}+ 1 \Big)^{1/p}.
\end{eqnarray*}
The left-hand side of the inequality can be estimated as follows:
\begin{eqnarray*}
    \|(z_1+ x,z_2)\|&=& \big(\|z_1+ x\|^r+ \|z_2\|^r \big)^{1/r}\\
    &\geq& \Big((1-\e)^r\big(\|z_1\|^p+ 1\big)^{r/p}+ \|z_2\|^r \Big)^{1/r}\\
    &\geq& (1-\e)\Big(\big(\|z_1\|^p+ 1\big)^{r/p}+ \|z_2\|^r \Big)^{1/r}.
\end{eqnarray*}

Setting $\alpha= \|z_1\|$ and $\beta= \|z_2\|$, we thus see that it suffices to prove the inequality
\begin{equation}\label{eq: integral Mink}
    \Big(\big(\alpha^p+ 1\big)^{r/p}+ \beta^r \Big)^{1/r}\geq \Big(\big(\alpha^r+ \beta^r\big)^{p/r}+ 1 \Big)^{1/p} \qquad(\alpha,\beta\geq0).
\end{equation}
This inequality can be proved via Minkowski integral inequality; however, we give a direct elementary proof. Consider the real function
\[ f(t)\coloneqq \Big(\big(\alpha^p+ t\big)^{r/p}+ \beta^r \Big)^{p/r}- \big(\alpha^r+ \beta^r\big)^{p/r} \qquad (t\geq 0). \]
Then, $f(0)=0$, and \eqref{eq: integral Mink} is equivalent to $f(1)\geq 1$. Hence, it suffices to prove that $f'(t)\geq 1$ for all $t\geq 0$, which can be checked directly. In fact,
\begin{eqnarray*}
    f'(t)&=& \frac{p}{r} \Big(\big(\alpha^p+ t\big)^{r/p}+ \beta^r \Big)^{\frac{p-r}{r}}\cdot \frac{r}{p} \big(\alpha^p+ t\big)^{\frac{r-p}{p}}\geq 1\\
    &\overset{p-r\geq0}{\iff}& \Big(\big(\alpha^p+ t\big)^{r/p}+ \beta^r \Big)^{1/r} \geq \big(\alpha^p+ t\big)^{1/p}.
\end{eqnarray*}
This last inequality is clearly satisfied, because $\beta\geq 0$.    
\end{proof}

Finally, we move to the main result of the section, where we give a lower bound for $\gamma(\X)$ when $\X$ is $(<\kappa)$-$\phi$-octahedral.

\begin{theorem}\label{thm: gamma of phi-octa} Let $\kappa$ be an infinite cardinal, $\phi$ a positive modulus, and $\X$ be a $(<\kappa)$-$\phi$-octahedral normed space. Then, for every normed space $\Y$ such that $\dens(\Y)< \cf(\kappa)$ and every $p\in [1,\infty]$ we have 
\[ \gamma(\X\oplus_p \Y)\geq \frac{2}{K(\X;\kappa)}. \]
\end{theorem}

\begin{proof} We split the argument in two steps, by first proving the particular case where $\Y=\{0\}$ and then bootstrapping the general case out of the particular one with the aid of \Cref{thm: phi-octa sum Y}. The argument we give in the first step is based upon a modification of the proof of \cite{Swanepoel}*{Proposition 2}.
\smallskip

\textbf{Step 1:} $\Y=\{0\}$.\\
Let $\X$ be a $(<\kappa)$-$\phi$-octahedral normed space and let  $\D$ be an arbitrary subset of $\X$ which is $2$-separated and $\rho$-dense. We let $R$ be the infimum of all these $\rho$ such that $\D$ is $\rho$-dense. Hence, by definition, for every $\e>0$, $\D$ is $(R+\e)$-dense and it is not $(R-\e)$-dense. We shall show that $R\geq \frac{2}{K(\X;\kappa)}$, which proves the desired inequality. Clearly, without any loss of generality, we can suppose that $R\leq 2$.

As $\phi$ is a positive modulus, for every $\delta>0$ we can pick $\e\in(0,1)$ such that 
\begin{equation}\label{eq: phi(delta/3)}
    1+ \phi\left(\frac\delta3\right)> \frac{R+\e}{(1-\e)(R-\e)}.
\end{equation}
Notice that $\e\to 0$ as $\delta\to 0$. Thus, if $\delta>0$ is small enough, we may suppose that $R+ \e+ \delta \leq3$. By definition, the set $\D$ is not $(R-\e)$-dense, whence there exists $x_0\in \X$ with the property that $\|d-x_0\|> R-\e$ for all $d\in \D$. Without loss of generality, we assume that $x_0=0$. As a consequence, $\|d\|> R-\e$ for all $d\in \D$. We now consider the set $\mathcal{Q}=\{d\in \D\colon \|d\|\leq R+\e+ \delta\}$.

\begin{claim}\label{claim: Q large} $|\mathcal{Q}|\geq \kappa$.
\end{claim}
\begin{proof}[Proof of \Cref{claim: Q large}] \renewcommand\qedsymbol{$\square$}
Suppose, towards a contradiction, that $|\mathcal{Q}|< \kappa$ and let $\Z\coloneqq \overline{\spn}(\mathcal{Q})$; then, $\dens(\Z)< \kappa$, therefore by definition of $(<\kappa)$-$\phi$-octahedrality there exists a vector $x\in S_\X$ such that, for every $d\in \mathcal{Q}$ and $\lambda \in \R$,
\[ \left\|\frac{d}{\|d\|}+\lambda x\right\|\geq (1-\e) \big(1+ \phi(|\lambda|)\big). \]
Hence, for $d\in \mathcal{Q}$, using that $R-\e< \|d\|\leq 3$ and that $\phi$ is non-decreasing, we get
\begin{eqnarray*}
	\|d- \delta x\| &=& \|d\|\left\|\frac{d}{\|d\|}-\frac{\delta}{\|d\|} x\right\|\\
	&>& (R-\e)(1-\e) \left(1+ \phi\left(\frac{\delta}{\|d\|}\right)\right)\\
	&\geq& (R-\e)(1-\e) \left(1+ \phi\left(\frac{\delta}{3}\right)\right) \overset{\eqref{eq: phi(delta/3)}}{>} R+\e.
\end{eqnarray*}

On the other hand, there exists $d_0\in \D$ such that $\|d_0-\delta x\|\leq R+\e$, because $\D$ is $(R+\e)$-dense. By the triangle inequality, $\|d_0\|\leq R+\e +\delta$, whence $d_0\in \mathcal{Q}$. This contradicts the fact that  $\|d- \delta x\|>R+\e$ whenever $d\in \mathcal{Q}$, thereby proving our claim.
\end{proof}

From the claim, we infer that the set $\mathcal{Q}$ is a $2$-separated subset of $(R+ \e+ \delta)B_\X$, having cardinality at least $\kappa$. By a rescaling and the definition of $K(\X;\kappa)$, we then obtain that
\[ 2\leq (R+\e +\delta) K(\X;\kappa). \]
Letting $\delta\to 0$ (hence, $\e\to 0$ too), we reach the inequality $R\geq \frac{2}{K(\X;\kappa)}$, which completes the proof of this step.
\smallskip

\textbf{Step 2:} The general case $\dens(\Y)< \cf(\kappa)$.\\
We first consider the case when $p<\infty$. \Cref{thm: phi-octa sum Y} yields that $\X\oplus_p \Y$ is $(<\kappa)$-$(\phi_p\circ \phi)$-octahedral. Therefore, application of the first step to $\X\oplus_p \Y$ leads us to
\[ \gamma(\X\oplus_p \Y)\geq \frac{2}{K(\X\oplus_p \Y;\kappa)}. \]
However, by \Cref{fact: K with a small summand}, $K(\X\oplus_p \Y;\kappa)= K(\X;\kappa)$, which concludes the proof in the case that $p<\infty$. Finally the case when $p=\infty$ just follows from the previous case by letting $p\to \infty$ and using that $\gamma(\X\oplus_p \Y)$ is a continuous function of $p$, by \Cref{fact: gamma and d_BM}.
\end{proof}

\subsection{Banach spaces with \texorpdfstring{$\gamma(\X) =2$}{gamma(X)=2}}\label{sec: gamma=2}
In this part we apply the results from this section to give examples of Banach spaces $\X$ such that $\gamma(\X)=2$. These spaces have the surprising property that any maximal packing already has the optimal covering property. Even though the result is now an almost direct application of the above theorems, we consider this as one of the most striking contributions of the paper.

\begin{theorem}\label{thm: gamma=2} Let $p\in [1,\infty)$ and let $(p_k)_{k=1}^\infty$ be any sequence in $[1, \infty)$ that diverges to $\infty$. Then, the Banach space
\[ \X_p= \left(\bigoplus_{k=1}^\infty \ell_{p_k}(\omega_k) \right)_{\ell_p} \]
satisfies $\gamma(\X_p)=2$.

In particular, $\X_1$ is an octahedral Banach space with $\gamma(\X_1)=2$; further, if $p>1$ and each $p_k\in (1,\infty)$, $\X_p$ is a reflexive Banach space with $\gamma(\X_p)=2$.
\end{theorem}

Incidentally, if $p=1$, or $p_k=1$ for some $k\in \N$, we have $K(\X_p)=2$, hence $\X_p$ provides yet another negative answer to \Cref{probl: gamma=2/k}. Compared to the previous examples, this is the only one where the product $\gamma(\X)\cdot K(\X)$ attains it maximum possible value, namely $4$.

\begin{proof} It is clear that $\X_p$ is reflexive when $p>1$ and $p_k \in(1,\infty)$. Likewise, $\X_1$, being an infinite $\ell_1$-sum, is octahedral. Therefore, we only have to prove that $\gamma(\X_p)= 2$.

Fix any real number $M$ and take $n\in \N$ such that $p_k\geq M$ for each $k\geq n$. We now consider the following Banach spaces:
\[ \Y\coloneqq \left(\bigoplus_{k=1}^{n-1} \ell_{p_k}(\omega_k) \right)_{\ell_p} \qquad\mbox{and}\qquad \Z\coloneqq \left(\bigoplus_{k=n+1}^\infty \ell_{p_k}(\omega_k) \right)_{\ell_p}, \]
which we canonically consider as subspaces of $\X_p$. Plainly, $\X_p= \Y\oplus_p \ell_{p_n}(\omega_n)\oplus_p \Z$.

We first prove the following estimate for the $\omega_n$-Kottman constant of $\ell_{p_n}(\omega_n)\oplus_p \Z$:
\begin{equation}\label{eq: K of tail part}
    K(\ell_{p_n}(\omega_n)\oplus_p \Z;\omega_n)\leq 2^{1/M}.
\end{equation}

Towards a contradiction, suppose that there exists a subset $\P$ of $B_{\ell_{p_n}(\omega_n)\oplus_p \Z}$ of cardinality $\omega_n$ that is $(2^{1/M}+ \e)$-separated. For $x\in \P$, let us write $x=(x(k))_{k=1}^\infty$, where $x(k)\in \ell_{p_k}(\omega_k)$ and let us write $\supp(x)\coloneqq \{k\in \N\colon x(k)\neq 0\}$. Up to replacing $\e$ with $\e/2$ and a small perturbation, we can assume that $\supp(x)$ is a finite set for each $x\in \P$. Since $\P$ has cardinality $\omega_n$ and $\N$ only has countably many finite subsets, there exists a countable subset $\P_0$ of $\P$ such that $\supp(x)\subseteq [n,N]$ for all $x\in \P_0$ (actually, one might even find such a subset $\P_0$ of cardinality $\omega_n)$. Hence, $\P_0$ is a $(2^{1/M}+ \e)$-separated subset of the unit ball of
\[ \bigoplus_{k=n}^N \ell_{p_k}(\omega_k). \]
However, the Kottman constant of said space equals $\max\{2^{1/p_k} \colon n\leq k\leq N\}$, which is at most $2^{1/M}$ because $p_k\geq M$ for all $k\geq n$. This contradicts the fact that $\P_0$ is $(2^{1/M}+ \e)$-separated and proves \eqref{eq: K of tail part}.

We now proceed with the proof. Since $\ell_{p_n}(\omega_n)$ is ${(<\omega_n)}$-$\phi_{p_n}$-octahedral by \Cref{lemma: ell_p octa}, \Cref{thm: phi-octa sum Y} implies that $\ell_{p_n}(\omega_n)\oplus_p \Z$ is $(<\omega_n)$-$\psi_n$-octahedral, where $\psi_n= \phi_p\circ \phi_{p_n}$. Thus, we are in position to apply \Cref{thm: gamma of phi-octa} to the spaces $\ell_{p_n}(\omega_n)\oplus_p \Z$ and $\Y$ (observing that $\dens(\Y)=\omega_{n-1}$ and $\omega_n$ is regular) to reach the conclusion that
\[ \gamma(\X_p)\geq\frac{2}{K(\ell_{p_n}(\omega_n)\oplus_p \Z;\omega_n)}\overset{\eqref{eq: K of tail part}}{\geq} \frac{2}{2^{1/M}}. \]
Since $M$ was arbitrary, we reach the conclusion that $\gamma(\X_p)=2$, and we are done.
\end{proof}

\begin{remark} Actually, one can directly conclude from \Cref{thm: phi-octa sum Y} that $\X_p$ itself is $(<\omega_n)$-$\psi_n$-octahedral for all $n$ and then slightly modify the above argument. In fact, \Cref{thm: gamma of phi-octa} yields the lower bound $\gamma(\X_p)\geq 2/K(\X_p;\omega_n)$ and then this Kottman constant is estimated via \Cref{fact: K with a small summand} and \eqref{eq: K of tail part}.
\end{remark}

\begin{remark} Let us also remark that, in order for our argument to work, we need $p_k\to \infty$ and index sets of larger and larger cardinality. Therefore, the argument above cannot produce a separable or super-reflexive example, \Cref{probl: separable gamma=2}. Here it is perhaps worth noticing that if one considers an increasing sequence of finite cardinals, the resulting space
\[ \X_p= \left(\bigoplus_{k=1}^\infty \ell_{p_k}^k \right)_{\ell_p} \]
is $\phi_p$-octahedral by \Cref{lemma: ell_p octa} and separable; further, $K(\X_p)= 2^{1/p}$. Therefore, $\gamma(\X_p)= \gamma^*(\X_p)= \frac{2}{2^{1/p}}$ by \Cref{thm: subgroup theta} and \Cref{thm: gamma of phi-octa}.
\end{remark}

In conclusion to this section, we notice that finite-dimensional spaces satisfy $\gamma(\X)<2$. This is a folklore fact, mentioned without proof, \emph{e.g.}, in \cite{Zong7/4}. Since at first it wasn't entirely clear to us how to prove it, we shall give the argument below.
\begin{fact}\label{fact: finite-dim gamma} For every finite-dimensional normed space $\X$ one has $\gamma(\X)< 2$.
\end{fact}

\begin{proof} Let $n$ be the dimension of $\X$ and fix a rank-$n$ lattice $\Lambda$ in $\X$ that is $2$-separated. Consider the torus $\mathbb{T}^n\coloneqq \X/ \Lambda$ with canonical projection $\pi\colon \X\to \mathbb{T}^n$ and its canonical quotient distance $d_{\mathbb{T}^n}(\pi(x), \pi(y))\coloneqq \dist(x-y, \Lambda)$. Take a maximal $2$-separated set $\P$ in $\mathbb{T}^n$; by compactness, $\P$ is $r$-dense in $\mathbb{T}^n$ for some $r<2$. Hence, $\pi^{-1}(\P)$ is $2$-separated (here we use that $\Lambda$ is also $2$-separated) and $r$-dense in $\X$, whence $\gamma(\X)\leq r<2$.
\end{proof}

\section{Applications to uniformly convex and Lebesgue spaces} \label{sec: UC and Leb}
In this section we combine the lower bound from \Cref{sec: subgroup} with the upper bound from \Cref{sec: phi-octahedral} and obtain two-sided estimates on the packing constants. We begin with a general result that pertains to all $(<\kappa)$-$\phi$-octahedral normed spaces and we then move to the main part of the section, about uniformly convex spaces. We introduce a modulus $\p_\X$, study its properties, and show that uniformly convex of density $\kappa$ spaces are $(<\kappa)$-$\p_\X$-octahedral. In particular, uniformly convex normed spaces $\X$ satisfy $\gamma(\X)> 1$ and $\gamma^*(\X)< 2$; this latter fact is particularly relevant because of \Cref{sec: gamma=2} and \Cref{probl: separable gamma=2}. Finally, in \Cref{sec: L_p} we specialise to Lebesgue spaces, where the constants can be computed and not merely estimated.

\begin{proposition}\label{prop: upper bound phi-octahedral} If $\X$ is a $(<\kappa)$-$\phi$-octahedral normed space of density $\kappa$, then
\[ \frac{2}{K(\X;\kappa)}\leq \gamma(\X)\leq \gamma^*(\X)\leq \frac{2}{1+\phi(1)}. \]
In particular, $\gamma^*(\X)< 2$.
\end{proposition}

\begin{proof} The lower bound is just a particular case of \Cref{thm: gamma of phi-octa}. For the upper bound, it is enough to observe that, for each $\e>0$, every $(<\kappa)$-$\phi$-octahedral normed space of density $\kappa$ satisfies the assumption of \Cref{thm: subgroup theta} with $\theta= (1-\e) (1+ \phi(1))$.
\end{proof}

We now move to uniformly convex spaces and we start by introducing a suitable modulus for them. Let us first consider  the \emph{duality map} $\J_\X\colon S_\X\to2^{{\X^*}}$  defined by
\[ \J_\X(x)\coloneqq \{x^*\in {\X^*}\setminus\{0\}\colon x^*(x)=\|x^*\|\}, \qquad x\in S_\X. \]

\begin{definition}\label{def: phi_X} Given a normed space $\X$, the \emph{tangential modulus of convexity} of $\X$ is the function $\p_\X\colon [0,\infty)\to [0,\infty)$ defined by
\[ \p_\X(t)\coloneqq \inf\{\|x+ tv\|-1\colon x\in S_\X, f\in \J_\X(x), v\in \ker(f)\cap S_\X\},\qquad t\in[0,\infty). \]
\end{definition}

Notice that, since the functional $f$ only intervenes via its kernel, it is equivalent to consider functionals such that additionally $\|f\|=1$. Let us further observe that we can also define the tangential modulus of convexity by means of Birkhoff-James orthogonality. Recall that $x \in\X$ is said to be \emph{Birkhoff-James orthogonal} to $y\in \X$ (denoted $x\perp_{\rm BJ} y$) if $\|x\|\leq \|x+\lambda y\|$ for every $\lambda \in\R$. It is immediate to see that
\[ \p_\X(t)= \inf\{\|x+ tv\|-1\colon x,v\in S_\X, x\perp_{\rm BJ} v \},\qquad t\in[0,\infty). \]
The modulus $\p_\X$ is a variation of Milman moduli of uniform convexity introduced and systematically investigated in \cite{Milman1971} (see also \cite{GoebelPrus}*{Chapter~7}). For example, if $\mathcal{H}$ is a Hilbert space (of dimension at least $2$), it is easy to compute directly that $\p_\mathcal{H}(t)= \sqrt{1+t^2}- 1$. The following proposition is quite standard and follows the same lines as in \cite{GoebelPrus}*{Chapter~7}.   

\begin{proposition}\label{prop: varphi modulus} For every normed space $\X$ the following assertions hold:
\begin{enumerate}
	\item\label{item: lip varphi} $\p_\X$ is $1$-Lipschitz and $\p_\X(0)=0$ (in particular, $\p_\X(t)\to 0$ as $t\to 0^+$);
    \item\label{item: varphi subspaces} if $\Y$ is a subspace of $\X$, then $\p_\X(t)\leq\p_\Y(t)$ for all $t\geq 0$;
    \item\label{item: non decreasing varphi}  $\p_\X(\lambda t)\leq \lambda \p_\X(t)$ for $\lambda\in (0,1)$ and $t\in[0,\infty)$ (thus, $\p_\X$ is non-decreasing);
    \item\label{item: ineq1 varphi} $\delta_\X\left(\frac{t}{1+ \p_\X(t)}\right)\leq \frac{\p_\X(t)}{1+ \p_\X(t)}$, whenever $t\in[0,2)$;
    \item\label{item: ineq2 varphi} if $t\in[0,2)$ and $\frac{t}{2}-2\delta_\X(t)\geq 0$, then $\p_\X(\frac{t}{2}- 2\delta_\X(t))\leq \delta_\X(t)$;
	\item\label{item: characterization varphi} $\p_\X$ is positive on $(0,\infty)$ if and only if $\delta_\X$ is positive on $(0,2)$;
    \item\label{item: varphi is a modulus} $\p_\X$ is a positive modulus if and only if $\X$ is uniformly convex.
\end{enumerate}
\end{proposition}

\begin{proof} \eqref{item: lip varphi} directly follows from the fact that $t\mapsto \|x+tv\|$ is $1$-Lipschitz, and \eqref{item: varphi subspaces} follows by extending functionals in $\J_\Y(x)$ to functionals in $\J_\X(x)$, where $x\in \Y$.

For \eqref{item: non decreasing varphi}, fix $x\in S_\X$, $f\in \J_\X(x)$, and $v\in \ker(f) \cap S_\X$. Then, we have
\[ \p_\X(\lambda t)\leq \|(1-\lambda)x+ \lambda(x+tv)\| -1\leq (1-\lambda)\|x\|+ \lambda\|x+ tv\|-1= \lambda(\|x+ tv\|-1). \]
Passing to the infimum, the conclusion follows.

We now prove \eqref{item: ineq1 varphi}. Notice first that, if $r>0$, $x,y\in r B_\X$, and $\|x-y\|\geq t$, then
\[ \left\|\frac{x+ y}{2} \right\|\leq r\left(1- \delta_\X \left(\frac{t}{r} \right)\right). \]
Now, fix $\eta>0$ and find $x\in S_\X$, $f\in \J_\X(x)$, and $v\in \ker(f) \cap S_\X$ such that $1+ \p_\X(t)+ \eta\geq \|x+ tv\|$. Applying the previous inequality to $x$ and $x+tv$, we obtain
\[ 1\leq \left\| \frac{x+ tv+ x}{2}\right \|\leq (1+\p_\X(t)+ \eta) \left(1- \delta_\X\left (\frac{t}{1+ \p_\X(t)+ \eta} \right)\right), \]
where the first inequality follows from the fact that $f(x+ \frac{t}{2}v)=1$. Letting $\eta\to 0^+$ and using the continuity of
$\delta_\X$ on $[0,2)$ we get
\[ 1\leq (1+\p_\X(t)) \left(1- \delta_\X\left (\frac{t}{1+ \p_\X(t)} \right)\right), \]
which directly leads to the desired inequality.

Next, we prove \eqref{item: ineq2 varphi}. Notice that if $\frac{t}{2}-2\delta_\X(t)=0$, the conclusion is trivial, since $\p_\X(0)=0$. Thus, we assume that $\frac{t}{2}-2\delta_\X(t)$ is positive. Let $\eta>0$ be such that $\frac{t}{2}- 2(\delta_\X(t)+ \eta)>0$. By definition of $\delta_\X(t)$, there exist $x,y\in S_\X$ so that $\|x-y\| =t$ and $1-\left\| \frac{x+ y}{2} \right\| <\delta_\X(t)+ \eta$. Define
\[ z\coloneqq \frac{x+y}{2}, \qquad z'\coloneqq \frac{z}{\|z\|}, \]
and take any $f\in\J_\X(z')$. Since $1- \delta_\X(t)- \eta \leq \|z\|\leq 1$, we see that $\|z- z'\|\leq \delta_\X(t)+ \eta$. Notice that $\frac{f(x)+ f(y)}{2}= \|z\|> 1-\delta_\X(t) -\eta$. Hence, without any loss of generality, we can assume that $f(x)> 1- \delta_\X(t)- \eta$. Consider the vector $x'\coloneqq x+ (1-f(x))z'$. Then, $f(x')= 1$ (hence, $x'- z'\in \ker(f)$) and $\|x-x'\|\leq \delta_\X(t) +\eta$. Next, we observe that
\begin{equation}\label{eq: x'- z' large}
    \|x'-z'\|\geq \|x-z\|- \|x-x'\|- \|z-z'\|\geq \frac{t}{2}-2\left(\delta_\X(t)+\eta\right)
\end{equation}
and that 
\[ \|x'\|\leq \|x\|+ \delta_\X(t) +\eta= 1+ \delta_\X(t) +\eta. \]
Finally, denote $t'\coloneqq \|x'-z'\|$ and $v\coloneqq \frac{x'-z'}{\|x'-z'\|}\in \ker(f) \cap S_\X$. By \eqref{eq: x'- z' large} and the fact that $\p_\X$ is non-decreasing, we obtain
\[ \p_\X\left( \frac{t}{2}- 2(\delta_\X(t)+ \eta) \right)\leq \p_\X(t')\leq \|z'+ t'v\|-1 =\|x'\| -1\leq \delta_\X(t) +\eta. \]
Letting $\eta\to 0^+$ and using the continuity of $\p_\X$, we obtain our inequality.

For the proof of \eqref{item: characterization varphi}, it is clear from \eqref{item: ineq1 varphi} that, if $\delta_\X$ is positive on $(0,2)$, then $\p_\X$ is positive on $(0,2)$, and hence on $(0,\infty)$ as well. For the converse implication, by Nordlander inequality, $\delta_\X(t)\leq 1-\sqrt{1-t^2/4}$; thus, there is $t_0\in (0,2)$ such that $\frac{t}{2}-2\delta_{\X}(t)>0$, whenever $t\in(0,t_0)$. The inequality in \eqref{item: ineq2 varphi} implies that, if $\p_\X$ is positive on $(0,\infty)$, then $\delta_\X$ is positive on  $(0,t_0)$ (and hence on $(0,2)$).

Finally, \eqref{item: varphi is a modulus} immediately follows from \eqref{item: lip varphi}, \eqref{item: non decreasing varphi}, \eqref{item: characterization varphi}, and the fact that $\X$ is uniformly convex if and only if $\delta_{\X}(t)>0$ for every $t\in(0,2)$.
\end{proof}

We will also need the fact that $\p_\X$ depends continuously on $\X$. A similar proof also gives that $\p_\X$ is invariant under taking the completion of $\X$ and we also record this, even if we don't require it. Since we need to approximate simultaneously a point and a functional, it is not surprising that we will use the Bishop--Phelps--Bollob\'as theorem, whose statement we recall (see, \emph{e.g.}, the proof of \cite{Phelps}*{Theorem~3.18}). \emph{Let $\X$ be a Banach space and $C\subseteq \X$ be a closed convex bounded set. Suppose that $\e>0$, $ f\in \X^*$, and $x\in C$ are such that $f(x)\geq \sup  f(C)-\e$. Then, for each $\alpha>0$, there exist $y\in C$ and $ g\in {{\X}^*}$ such that
\[ \|x-y\|\leq \frac{\e}{\alpha}, \qquad \|f- g\|\leq \alpha, \qquad g(y)=\sup g(C). \]}

\begin{fact}\label{fact: varphi d_BM cont} For every $t\geq 0$, the map $\X\mapsto \p_\X(t)$ is continuous with respect to the Banach--Mazur distance. Further, $\p_\X(t)= \p_{\hat{\X}}(t)$, where $\hat{\X}$ is the completion of $\X$.
\end{fact}

\begin{proof} To prove the first assertion, let $t\geq0$, $\e\in(0,\frac14)$, and let $(\Y,\nn\cdot)$ be a renorming of $\X$ such that $B_\X\subseteq B_\Y\subseteq (1+\e^2)B_\X$. Find $x\in S_\X$, $f_0\in \J_\X(x)$, and $v\in \ker(f_0) \cap S_\X$ such that $ \|x+ tv\|-1\leq \p_\X(t)+ \e$. Without loss of generality, we assume that $\|f_0\| =1$. Consider the 2-dimensional subspaces $\V\coloneqq (\spn\{x,v\},\n)$ and $\W\coloneqq (\spn\{x,v\},\nn\cdot)$ of $\X$ and $\Y$ respectively. In particular, $B_\V\subseteq B_\W\subseteq (1+\e^2)B_\V$. If we denote $f\coloneqq f_0\cut_\V\in S_{\V^*}$, then $\sup f(B_\W)\leq 1+\e^2$, hence $f(x)\geq\sup f(B_\W)-\e^2$. 

An application of the Bishop--Phelps--Bollob\'as theorem (to the set $B_\W$, as a closed convex bounded subset of $\V$, and the point $x\in B_\W$) yields the existence of $y\in B_\W$ and $g\in \V^*$ such that
\begin{equation}\label{eq: BPB}
    \|x-y\|\leq \e, \qquad \|f- g\|\leq \e, \qquad g(y)=\sup g(B_\W).
\end{equation}
Observe that $g(y)= \nn{g}\geq \|g\|\geq \|f\|-\e \geq\frac{1}{2}$; thus, $g$ is non-zero, which implies that necessarily $y\in S_{\W}$ and $g\in \J_\W(y)$. Define $w'\coloneqq v-\frac{g(v)}{g(y)}y$. Then $w'\in \ker(g)$ and
\[ \|v-w'\|\leq \left|\frac{g(v)}{g(y)} \right| \|y\| \leq 2(1+\e^2)|g(v)|\leq 4|(g- f)(v)|\leq 4\e, \]
implying in particular that $w'\neq 0$ (as $\e<\frac{1}{4}$). Further, $1-4\e\leq \|w'\|\leq 1+4\e$, thus, if we define $w''\coloneqq \frac{w'}{\|w'\|}$, we obtain that $\|w'- w''\|\leq 4\e$. Moreover, setting $w\coloneqq \frac{w'}{\nn{w'}}$, we clearly have $\|w-w''\|\leq \e^2$. As a consequence, 
\[ \|v-w\|\leq \|v-w'\|+ \|w'-w''\|+ \|w''-w\|\leq 4\e+ 4\e+ \e^2\leq 9\e. \] 
Finally, we obtain
\begin{eqnarray*}
    \p_{\Y}(t) \leq \p_{\W}(t)\leq \nn{y+tw} -1&\leq& \|y+tw\| -1\\
    &\leq& \|x+tv\|-1+\|y-x\|+t\|v-w\|\\
    &\leq& \p_{\X}(t)+2\e+9 t\e.
\end{eqnarray*}
For the converse inequality, let $\Y'$ be the renorming of $\Y$ such that $B_{\Y'}=\frac{1}{1+\e^2}B_\Y$. Then,  $B_{\Y'}\subseteq B_\X\subseteq (1+\e^2)B_{\Y'}$, and the previous part gives $\p_{\X}(t) \leq \p_{\Y'}(t)+2\e+9 t\e$. However, $\p_{\Y'}=\p_{\Y}$, as $\Y$ and $\Y'$ are isometric. Thus,
\[ |\p_{\X}(t)-\p_{\Y}(t)| \leq 2\e+9 t\e, \]
which proves the first clause.

For the second part we proceed similarly. First, observe that $\p_{\hat{\X}}(t)\leq \p_\X(t)$ by \Cref{prop: varphi modulus}\eqref{item: varphi subspaces}. To prove the other inequality, fix $\e\in(0,\frac14)$ and find $\hat{x}\in S_{\hat{\X}}$, $\hat{f}\in \J_{\hat{\X}}(\hat{x})$, and $\hat{v}\in \ker(\hat{f}) \cap S_{\hat{\X}}$ such that $\|\hat{x}+ t\hat{v}\|-1\leq \p_{\hat{\X}}(t)+ \e$. Up to a scaling, we assume that $\|\hat{f}\| =1$. Let us also find $x,v\in S_\X$ such that $\|\hat{x}- x\|\leq \e^2$ and $\|\hat{v}- v\|\leq \e$. Denote  $\W\coloneqq \spn\{x,v\}$ and $f\coloneqq  \hat{f}\cut_\W\in {{\W}^*}$. We clearly have 
\[ f(x)\geq \hat{f}(\hat{x})-\e^2\geq \sup f(B_{\W})-\e^2; \]
in particular, $\|f\|\geq 1-\e^2\geq 1-\e$. By the Bishop--Phelps--Bollob\'as theorem there exist $y\in B_\W$ and $g\in\W^*$ such that \eqref{eq: BPB} holds. Then, $\|g\|\geq \|f\|- \e\geq 1- 2\e\geq \frac{1}{2}$. So, $g$ is non-zero, and we have $y\in S_\W$ and $g\in \J_\W(y)$.

If we define $w'=v- \frac{g(v)}{g(y)}y$, we have that $w'\in \ker(g)$ and 
\[ \|v- w'\|= \frac{|g(v)|}{|g(y)|}\leq 2|g(v)|\leq 2\big( |g(v)- f(v)|+ |f(v)- \hat{f}(\hat{v})|\big) \leq 4\e. \]
As $\e< \frac{1}{4}$ and $\|v\|=1$, we get that $w'\neq 0$. Further, setting $w\coloneqq \frac{w'}{\|w'\|}$, we also have $\|w- w'\|\leq 4\e$, whence $\|v- w\|\leq 8\e$. Noting that $\|\hat{v}- w\|\leq 9\e$ and $\|\hat{x}- y\|\leq 2\e$, we finally conclude that
\[ \p_\X(t)\leq \p_\W(t)\leq \|y+tw\|-1\leq \|\hat{x}+ t\hat{v}\|-1+ 2\e+ 9t\e\leq \p_{\hat{\X}}(t)+ 3\e+ 9t\e. \]
Letting $\e\to 0^+$, we get the desired inequality, and  we are done.
\end{proof}

We are now in position to prove our main results on uniformly convex spaces.
\begin{theorem}\label{thm: UC implies phi-octa} Uniformly convex normed spaces $\X$ of density $\kappa$ are $(<\kappa)$-$\p_\X$-octahedral. 
\end{theorem}

\begin{proof} Let $\Z$ be a closed subspace of $\X$ such that $\kappa_0\coloneqq \dens(\Z)<\kappa$, and $\e>0$. Take a subset $W$ of $S_\Z$ that is $\e$-dense and has cardinality $\kappa_0$; for each $w\in W$ pick an element $f_w$ in $\J_\X(w)$. Since $\X$ is, in particular, reflexive,  $w^*$-$\dens(\X^*)= \kappa$; therefore, there exists
\[ x\in S_\X\cap \bigcap_{w\in W}\ker(f_w). \]
By definition of $\p_\X$, for all $w\in W$, we have $\p_\X(|\lambda|)\leq \|w+ \lambda x\| -1$; hence,
\[ \|w+\lambda x\|\geq 1+\p_\X(|\lambda|), \qquad \mbox{for all }\, w\in W \mbox{ and } \lambda\in \R. \]
Now, if $z\in S_\Z$, we can find $w\in W$ such that $\|z-w\|\leq \e$, and we have
\[ \|z+ \lambda x\|\geq \|w+ \lambda x\|-\e\geq 1+ \p_\X(|\lambda|)-\e\geq (1-\e)\big(1+ \p_\X(|\lambda|)), \]
which proves that $\X$ is $(<\kappa)$-$\p_\X$-octahedral. 
\end{proof}

\begin{theorem}\label{thm: superreflexive and not} Let $\X$ and $\Y$ be normed spaces with $\dens(\Y)< \dens(\X)$ and $p\in [1,\infty)$. Then, for every infinite cardinal $\kappa$ such that $\dens(\Y)< \cf(\kappa)$ and $\kappa\leq \dens(\X)$, we have:
\begin{enumerate}
    \item\label{item: bound in superreflexive} If $\hat{\X}$ (the  completion of $\X$) is super-reflexive, then
	\begin{equation}\label{eq: gamma VS delta_X in UC}
	    \frac{1}{1-\delta_\X(1)}\leq \frac{2}{K(\X;\kappa)}\leq \gamma(\X\oplus_p \Y)\leq \gamma^*(\X\oplus_p \Y)\leq \frac{2}{1+\phi_p\circ\p_\X(1)}.
    \end{equation}
    In particular, if $\X$ is uniformly convex, $1< \gamma(\X\oplus_p \Y)\leq \gamma^*(\X\oplus_p \Y)< 2$.
    \item\label{item: bound} As a consequence (without assuming that $\hat{\X}$ is super-reflexive), we obtain:
    \begin{equation}\label{eq: gamma VS delta_X in general} 
        \frac{1}{1-\delta_\X(1)}\leq \gamma(\X\oplus_p \Y)\leq \gamma^*(\X\oplus_p \Y)\leq
        \frac{2}{1+\phi_p\circ\p_\X(1)}.
    \end{equation}
\end{enumerate}
In either clause, the lower bound also holds when $p=\infty$.
\end{theorem}

\begin{proof} To prove \eqref{item: bound in superreflexive}, suppose first that $\X$ is uniformly convex and let $\kappa_0\geq \kappa$ be the density character of $\X$. Then, $\X$ is $(<\kappa_0)$-$\p_\X$-octahedral, by \Cref{thm: UC implies phi-octa}. Hence, it is $(<\kappa)$-$\p_\X$-octahedral as well, and the second inequality in \eqref{eq: gamma VS delta_X in UC} follows directly from \Cref{thm: gamma of phi-octa}. Further, \Cref{thm: phi-octa sum Y} implies that $\X\oplus_p \Y$ is $(<\kappa_0)$-$(\phi_p\circ \p_\X)$-octahedral, thus the last inequality in \eqref{eq: gamma VS delta_X in UC} is a consequence of \Cref{prop: upper bound phi-octahedral}. Therefore, for uniformly convex $\X$, we obtained
\[ \frac{2}{K(\X;\kappa)}\leq \gamma(\X\oplus_p \Y)\leq \gamma^*(\X\oplus_p \Y)\leq \frac{2}{1+\phi_p\circ\p_\X(1)}. \]

For the general case, if $\hat{\X}$ is a super-reflexive space, the set of uniformly convex norms on $\X$ is dense in the set of all equivalent norms (because this is true in $\hat{\X}$). This and the continuity of $K(\cdot;\kappa)$, $\gamma(\cdot)$, and $\X\mapsto \p_\X(1)$ with respect to the Banach-Mazur distance (by \Cref{fact: K d_BM cont}, \Cref{fact: gamma and d_BM}, and \Cref{fact: varphi d_BM cont} respectively) imply that the above inequalities are true for all  spaces with super-reflexive completion as well. To complete the proof of \eqref{item: bound in superreflexive}, it is just sufficient to observe that $K(\X;\kappa)\leq K(\X)\leq 2(1-\delta_\X(1))$, where the second inequality is due to Maluta and Papini, \cite{MalutaPapini}*{Theorem~2.6}. 

We now prove \eqref{item: bound}, and we begin with the first inequality in \eqref{eq: gamma VS delta_X in general}. If $\delta_\X(1)=0$, then the inequality is trivially valid. So, we can assume that $\delta_{\hat{\X}}(1)=\delta_\X(1)>0$ and hence that $\hat{\X}$ is uniformly non-square, in particular super-reflexive, \cite{James}. Therefore, the inequality follows from \eqref{item: bound in superreflexive}. The last inequality in \eqref{eq: gamma VS delta_X in general} is proved similarly. If $\p_\X(1)= 0$, it reduces to the fact that $\gamma^*(\X) \leq2$ for all spaces. Hence, we can assume that $\p_\X(1)>0$; thus, by continuity there exists $t\in (0,1)$ such that $\p_\X(t)>0$. We claim that $\delta_\X (2t)>0$. In fact, if $\delta_\X(2t)$ were null, \Cref{prop: varphi modulus}\eqref{item: ineq2 varphi} would give $\p_\X(t)\leq \delta_\X(2t) =0$, which is a contradiction. Finally, since $\delta_{\hat{\X}}(2t)= \delta_\X(2t) >0$, $\hat{\X}$ is super-reflexive, and the inequality follows from \eqref{item: bound in superreflexive}. 
\end{proof}

In the above theorem it is also possible to give an upper bound for $\gamma^*(\X)$ that involves the modulus of convexity $\delta_\X$, instead of $\p_\X$. Let $t_\X\coloneqq \sup\{t\in[0,2)\colon \delta_\X(t)\leq 1-t \}$ and observe that, since $\delta_\X$ is non-decreasing and continuous on $[0,2)$, $t_\X$ is the unique $t\in[0,1]$ such that $\delta_\X(t)= 1-t$. Plainly, if $\delta_\X(1)> 1$ (\emph{e.g.}, if $\X$ is uniformly convex) then $t_\X<1$. Moreover, by Nordlander's inequality, we get
\[ 1-t_\X= \delta_\X(t_\X)\leq 1-\sqrt{1- \frac{{t_\X}^2}{4}}, \]
which yields $t_\X\geq \frac{2}{\sqrt{5}}$.

\begin{corollary} For every infinite-dimensional normed space $\X$, we have
\[ \gamma^*(\X)\leq 2t_\X\leq 2\left(1-\delta_\X \left(\frac{2}{\sqrt{5}}\right) \right). \]
\end{corollary}

However, this bound is in general weaker than the one from \Cref{thm: superreflexive and not}. For example, if $\mathcal{H}$ is a Hilbert space, then
\[ \frac{1}{1+\p_\mathcal{H}(1)}= \frac{1}{\sqrt{2}}< \frac{2}{\sqrt{5}}= t_\mathcal{H}. \]

\begin{proof} \Cref{prop: varphi modulus}\eqref{item: ineq1 varphi} yields
\[ \delta_\X\left(\frac{1}{1+\p_\X(1)}\right)\leq \frac{\p_\X(1)}{1+ \p_\X(1)}= 1-\frac{1}{1+\p_\X(1)}; \]
hence, by the definition of $t_\X$, we conclude that $\frac{1}{1+\p_\X(1)}\leq t_\X$. \Cref{thm: superreflexive and not} then gives
\[ \gamma^*(\X)\leq\frac{2}{1+\p_\X(1)}\leq 2 t_\X= 2(1-\delta_\X(t_\X))\leq 2\left(1-\delta_\X \left(\frac{2}{\sqrt{5}}\right) \right). \qedhere \]
\end{proof}

While \Cref{thm: superreflexive and not}\eqref{item: bound in superreflexive} implies that each uniformly convex space $\X$ satisfies $\gamma(\X)>1$, it doesn't allow to deduce the same assertion for general super-reflexive spaces, as there are super-reflexive spaces $\X$ with $\delta_\X(1)=0$. Actually, by the main result in \cite{DRShilbert}, if $\kappa^\omega= \kappa$ there is a Banach space $\X$, isomorphic to $\ell_2(\kappa)$, that admits a lattice tiling by balls; hence $\gamma^*(\X)=1$. We now show that a variation of the argument even gives a separable example.

\begin{proposition}\label{prop: renorm ell_2} For every infinite cardinal $\kappa$ there is an equivalent norm $\nn\cdot$ on $\ell_2(\kappa)$ such that $\gamma^*((\ell_2(\kappa),\nn\cdot))= 1$. In particular, there are (infinite-dimensional) separable super-reflexive Banach spaces with $\gamma^*(\X)= 1$.
\end{proposition}

\begin{proof} According to \cite{DRShilbert}*{Corollary 3.4}, $\ell_2(\kappa)$ contains a subgroup $\D$ that is $\sqrt{2}$-separated and such that $\dist(x,\D)\leq 1$ for all $x\in \ell_2(\kappa)$. Consider the Voronoi cells $\{V_d\}_{d\in \D}$ associated to $\D$. By the argument in \cite{DRShilbert}*{Proposition 2.3}, $V_0$ is a bounded, symmetric convex body and $V_d= d+ V_0$ for all $d\in \D$. Further, the Voronoi cells are non-overlapping, because, for each $d\in \D$, the hyperplane
\[ \left\{x\in \ell_2(\kappa)\colon \langle x,d\rangle= \tfrac{1}{2}\|d\|^2\right\} \]
separates $V_0$ from $V_d$. Consequently, if we let $\nn\cdot$ be the equivalent norm on $\ell_2(\kappa)$ whose unit ball is $V_0$, we see that $\{V_d\}_{d\in \D}$ is a packing by unit balls in $(\ell_2(\kappa), \nn\cdot)$. Finally, we show that $\bigcup_{d\in \D}V_d$ is dense in $\ell_2(\kappa)$, which plainly yields $\gamma^*((\ell_2(\kappa),\nn\cdot))= 1$.

For this, take any $x\in \ell_2(\kappa)$ such that $\dist(x, \D)< 1$ and pick $d\in \D$ such that $r\coloneqq \|x- d\| <1$. Thus the set $\D\cap B(x,r)$ is non-empty and $\sqrt{2}$-separated. As $K(\ell_2(\kappa))= \sqrt{2}$ and $r<1$, we deduce that $\D\cap B(x,r)$ is a finite set. Therefore, there exists $h\in \D$ that minimises the distance of $x$ to $\D$, whence $x\in V_h$. Finally, since $\dist(x,\D)\leq 1$ for all $x\in \ell_2(\kappa)$, we see that the set of $x\in \ell_2(\kappa)$ with $\dist(x,\D)< 1$ is dense in $\ell_2(\kappa)$. A fortiori, $\bigcup_{d\in \D}V_d$ is dense in $\ell_2(\kappa)$, as desired.
\end{proof}

\subsection{Lebesgue spaces}\label{sec: L_p}
We now specialise the results from this section to compute $\gamma$ and $\gamma^*$ for the $\ell_p(\kappa)$ and the $L_p(\mu)$ spaces and some direct sums thereof. As a consequence, we generalise Swanepoel's result \cite{Swanepoel} that $\gamma(\ell_p)= \frac{2}{2^{1/p}}$ and the result from \cite{DRShilbert}*{Corollary 3.4} that $\gamma^*(\ell_p(\kappa))= 2/2^{1/p}$; we also obtain more counterexamples to \Cref{probl: gamma=2/k}. The lower bounds follow directly from \Cref{thm: superreflexive and not}, while the upper bounds require the modulus $\phi_p$ and \Cref{prop: p-octa sum Y}. We begin with the purely atomic case.

\begin{theorem}\label{thm: gamma ell_p} Let $p\in[1,\infty)$, $\kappa$ be an infinite cardinal, and $\Y$ a normed space with $\dens(\Y)< \kappa$. Then:
\begin{enumerate}
    \item\label{item: ell_p small r} For $1\leq r\leq p$
    \[ \gamma(\ell_p(\kappa)\oplus_r \Y)= \gamma^*(\ell_p(\kappa)\oplus_r \Y)= \frac{2}{2^{1/p}}. \]
    \item\label{item: ell_p large r} For $p\leq r<\infty$
    \[ \frac{2}{2^{1/p}}\leq \gamma(\ell_p(\kappa) \oplus_r \Y)\leq \gamma^*(\ell_p(\kappa) \oplus_r \Y)\leq \frac{2}{2^{1/r}}. \]
    \item\label{item: ell_p infty r} For $r=\infty$
    \[ \frac{2}{2^{1/p}}\leq \gamma(\ell_p(\kappa) \oplus_\infty \Y)\leq \gamma^*(\ell_p(\kappa) \oplus_\infty \Y)\leq \max\left\{ \frac{2}{2^{1/p}}, \gamma^*(\Y)\right\}. \]
\end{enumerate}
\end{theorem}

\begin{proof} We begin by proving the three lower bounds for $\gamma(\ell_p(\kappa) \oplus_r \Y)$. Since $\dens(\Y)< \kappa$, there is a regular cardinal $\kappa_0$ such that $\dens(\Y)< \kappa_0\leq \kappa$ (to wit, one might pick the successor of $\dens(\Y)$). Further, $K(\ell_p(\kappa); \kappa_0)\leq K(\ell_p(\kappa))= 2^{1/p}$. When $p>1$, $\ell_p(\kappa)$ is uniformly convex, thus \Cref{thm: superreflexive and not} implies
\[ \frac{2}{2^{1/p}}\leq \gamma(\ell_p(\kappa) \oplus_r \Y) \qquad \mbox{for all }r\in [1,\infty]. \]
This inequality is also trivially true when $p=1$, thereby proving the lower bounds.

Now, to the upper bounds, where we use the fact that $\ell_p(\kappa)$ is $(<\kappa)$-$\phi_p$-octahedral, by \Cref{lemma: ell_p octa}. If $r\leq p$, \Cref{prop: p-octa sum Y} shows that $\ell_p(\kappa)\oplus_r \Y$ is $(<\kappa)$-$\phi_p$-octahedral; thus \Cref{prop: upper bound phi-octahedral} yields $\gamma^*(\ell_p(\kappa)\oplus_r \Y)\leq \frac{2}{2^{1/p}}$, and proves \eqref{item: ell_p small r}. Likewise, if $p\leq r< \infty$, $\ell_p(\kappa)$ is $(<\kappa)$-$\phi_r$-octahedral, whence $\ell_p(\kappa) \oplus_r \Y$ is $(<\kappa)$-$\phi_r$-octahedral as well, again by \Cref{prop: p-octa sum Y}. Thus, \Cref{prop: upper bound phi-octahedral} also proves the upper bound in \eqref{item: ell_p large r}. Finally, the upper bound in \eqref{item: ell_p infty r} is just \Cref{fact: gamma and d_BM} (where $\gamma^*(\ell_p(\kappa))= \frac{2}{2^{1/p}}$ by \eqref{item: ell_p small r}).
\end{proof}

We now explicitly distil a particular case of the previous theorem, which leads us to one more counterexample to \Cref{probl: gamma=2/k}; further, it yields an example of octahedral Banach space $\X$ with $\gamma(\X)>1$, which compares to \Cref{thm: octahedral}.
\begin{example}\label{ex: octa gamma>1} Consider the Banach space $\X= \ell_1\oplus_1 \ell_p(\omega_1)$, for $p>1$. Then, $\gamma(\X)= \gamma^*(\X)= \frac{2}{2^{1/p}}$, while $K(\X)=2$. Comparing to the examples in \Cref{sec: packing LUR}, in this example we can compute exactly both $\gamma(\X)$ and $K(\X)$, instead of merely estimating them. Moreover, the Banach space $\X$ is octahedral, yielding an example of a non-separable octahedral Banach space with $\gamma(\X)>1$ (and actually, $\gamma(\X)$ can be chosen as close to $2$ as we wish by taking sufficiently large $p$). 
\end{example}

Next, we move to the function spaces. The final result is weaker than the previous one, because $K(L_p(\mu))= \max\{2^{1/p}, 2^{1/q}\}$ if $\mu$ is not purely atomic, \cite{WellWill}*{Theorem 16.9} (and where $q$ is the conjugate index to $p$). On the other hand, the proof is essentially identical, the only difference being that the proof that $L_p(\mu)$ is $(<\kappa)$-$\phi_p$-octahedral is more complicated.

\begin{proposition}\label{prop: L_p p-octa} For every measure space $(\M,\Sigma, \mu)$ and $p\in [1,\infty)$, the space $L_p(\mu)$ is $(<\kappa)$-$\phi_p$-octahedral, where $\kappa$ is the density of $L_p(\mu)$.
\end{proposition}

\begin{proof} We first prove the claim for $\mu$ being the product measure on $\{-1,1\}^\kappa$, and then deduce the general case, by means of Maharam's theorem.
\smallskip

\textbf{Step 1:} The case of $L_p(\{-1,1\}^\kappa)$.\\
We begin by recalling some well-known facts concerning functions in $L_p(\{-1,1\}^\kappa)$, see, \emph{e.g.}, \cites{Argyros, EnfloRosenthal, Kosz_TAMS}. For a function $f\colon \{-1,1\}^\kappa \to \R$ and a subset $\Lambda$ of $\kappa$, $f$ \emph{depends on $\Lambda$} if there is a function $g\colon \{-1,1\}^\Lambda\to \R$ such that $f(x)= g(x\cut_\Lambda)$ for all $x\in \{-1, 1\}^\kappa$ (equivalently, $f(x)= f(y)$ whenever $x\cut_\Lambda= y\cut_\Lambda)$. The function $f$ \emph{depends on finitely many} (resp.~\emph{countably many}) \emph{coordinates} if there is a finite (resp.~countable) subset $\Lambda$ of $\kappa$ such that $f$ depends on $\Lambda$. By the Stone--Weierstrass theorem, the set of functions that depend on finitely many coordinates is dense in $\C(\{-1,1\}^\kappa)$ (note that all such functions are continuous). Hence, by Lusin's theorem, such a set is also dense in $L_p(\{-1,1\}^\kappa)$. Since the set of functions that depend on a given set $\Lambda$ is closed under pointwise limits, it follows that every function in $L_p(\{-1,1\}^\kappa)$ depends on countably many coordinates.

We can now begin the proof. Fix $\e>0$ and a subspace $\Z$ of $L_p(\{-1,1\}^\kappa)$ such that $\dens(\Z)< \kappa$. As we saw in \Cref{ex: phi_p}, it is sufficient to find $g\in L_p(\{-1,1\}^\kappa)$ with $\|g\| =1$ and with the property that
\begin{equation}\label{eq: L_p p-octa}
    \|f+ g\|\geq (1-\e)\big(\|f\|^p+ 1\big)^{1/p} \qquad \mbox{for all }f\in \Z.
\end{equation}

We first consider the case that $\kappa= \omega$. In this case, both $\{-1,1\}^\omega$ and $[0,\infty)$ (with Lebesgue measure) are separable non-atomic measure algebras. Hence Carath\'eodory's theorem implies that $L_p(\{-1,1\}^\omega)$ is isometric to $L_p([0,\infty))$, and we actually perform the argument in the latter space. To begin with, as in the proof of \Cref{lemma: ell_p octa}, there is $M>0$ such that \eqref{eq: L_p p-octa} trivially holds for all $f\in \Z$ with $\|f\|>M$. Consider the function $g\coloneqq \bone_{[n,n+1]}$; for a fixed $f\in \Z$ and $n$ large enough
\[ \int_0^\infty |f+ g|^p d\mu= \int_0^n |f|^p d\mu+ \int_n^{n+1} |f+1|^p d\mu+ \int_{n+1}^\infty |f|^p d\mu\geq (1-\e/2)^p\big(\|f\|^p+ 1\big). \]
As the set $\{f\in \Z\colon \|f\|\leq M\}$ is compact (since $\Z$ is finite-dimensional), there is $n\in \N$ so that the above inequality (with $\e/2$ replaced by $\e$) holds for all $f\in \Z$ with $\|f\|\leq M$, which proves \eqref{eq: L_p p-octa}.

Next, we consider the case that $\kappa$ is uncountable. Since every function in $L_p(\{-1,1\}^\kappa)$ depends on countably many coordinates and $\dens(\Z)< \kappa$, there exists a subset $\Lambda$ of $\kappa$ with $|\Lambda|< \kappa$ such that all functions in $\Z$ only depend on the coordinates from $\Lambda$. Fix $n\in \N$ large enough that
\begin{equation}\label{eq: n large for L_p}
    \big(1- 2^{-n} \big)^{1/p}- 2^{-n/p}\geq 1- \e
\end{equation}
and choose ordinals $\alpha_1,\dots, \alpha_n\in \kappa \setminus \Lambda$. For a sign $\sigma \in\{-1,1\}^n$ consider the set
\[ A_\sigma\coloneqq \{x\in \{-1,1\}^\kappa\colon x(\alpha_j)= \sigma(j),\; j=1,\dots, n \}. \]
These sets form a partition of $\{-1,1\}^\kappa$ into sets of measure $2^{-n}$. Notice that, for a function $f\in \Z$,
\[ \int_{A_\sigma}|f|^p d\mu= 2^{-n}\int_{\{-1,1\}^\kappa} |f|^p d\mu \]
since $f$ does not depend on the coordinates $\alpha_1,\dots, \alpha_n$; in other words, $\|f\cdot \bone_{A_\sigma}\|= 2^{-n/p} \|f\|$.
We are now in position to define the function $g$ we are after. Fix one sign $\sigma$ and define $g\coloneqq 2^{n/p}\bone_{A_\sigma}$; note that $\|g\|= 1$. Then, for every $f\in \Z$ we have:
\begin{eqnarray*}
    \|f+ g\| &\geq& \|f\cdot \bone_{A_\sigma^\complement}+ g\|- \|f\cdot \bone_{A_\sigma}\|\\
    &=& \big(\|f\cdot \bone_{A_\sigma^\complement} \|^p + 1\big)^{1/p}- 2^{-n/p} \|f\|\\
    &=& \Big( \big(1- 2^{-n} \big) \|f\|^p + 1 \Big)^{1/p}- 2^{-n/p} \|f\|\\
    &\geq& \big(1- 2^{-n} \big)^{1/p}\cdot \big(\|f\|^p + 1 \big)^{1/p}- 2^{-n/p} \big(\|f\|^p + 1 \big)^{1/p}\\
    &=& \Big( \big(1- 2^{-n} \big)^{1/p}- 2^{-n/p}\Big)\cdot \big(\|f\|^p + 1 \big)^{1/p}\\
    &\overset{\eqref{eq: n large for L_p}}{\geq}& \big(1- \e \big)\cdot \big(\|f\|^p + 1 \big)^{1/p}
\end{eqnarray*}
This proves \eqref{eq: L_p p-octa} and concludes the first step.
\smallskip

\textbf{Step 2:} The general case.\\
By Maharam's theorem (see, \emph{e.g.}, \cite{Semadeni}*{\S~26} or the proof of \cite{LiTzaII}*{Theorem 1.b.2}) there are a cardinal $\kappa_0$ and a family $(\kappa_i)_{i\in I}$ of infinite cardinals such that
\[ L_p(\mu)\equiv \ell_p(\kappa_0)\oplus_p \left( \bigoplus_{i\in I}L_p(\{-1,1\}^{\kappa_i}) \right)_{\ell_p}. \]
Suppose first that $\kappa= \sup_{i\in I}\kappa_i$. Then, $L_p(\{-1,1\}^{\kappa_i})$ is $(<\kappa_i)$-$\phi_p$-octahedral by the previous step, hence \Cref{prop: p-octa sum Y} yields that $L_p(\mu)$ is also $(<\kappa_i)$-$\phi_p$-octahedral for every $i\in I$. Since $\kappa= \sup_{i\in I}\kappa_i$, this readily implies that $L_p(\mu)$ is $(<\kappa)$-$\phi_p$-octahedral as well (if fact, if $\dens(\Z)< \kappa$, there is some $i\in I$ with $\dens(\Z)< \kappa_i$). Otherwise, if $\sup_{i\in I}\kappa_i< \kappa$, it necessarily follows that $|I|= \kappa$. Therefore, we may write 
\[ L_p(\mu) \equiv \left( \bigoplus_{\alpha< \kappa} \X_\alpha \right)_{\ell_p}, \]
where each $\X_\alpha\neq \{0\}$. Thus, \Cref{lemma: ell_p octa} yields that $L_p(\mu)$ is $(<\kappa)$-$\phi_p$-octahedral.
\end{proof}

\begin{theorem}\label{thm: gamma L_p} Let $p\in[1,\infty)$, $\kappa$ be an infinite cardinal, $(\M,\Sigma, \mu)$ a measure space such that $\dens(L_p(\mu))= \kappa$, and $\Y$ a normed space with $\dens(\Y)< \kappa$. Then:
\begin{enumerate}
    \item\label{item: L_p small r} For $1\leq r\leq p$
    \[ \min\left\{\frac{2}{2^{1/p}}, \frac{2}{2^{1/q}}\right\}\leq \gamma(L_p(\mu)\oplus_r \Y)\leq \gamma^*(L_p(\mu)\oplus_r \Y)\leq \frac{2}{2^{1/p}}. \]
    \item\label{item: L_p large r} For $p\leq r<\infty$
    \[ \min\left\{\frac{2}{2^{1/p}}, \frac{2}{2^{1/q}}\right\}\leq \gamma(L_p(\mu) \oplus_r \Y)\leq \gamma^*(L_p(\mu) \oplus_r \Y)\leq \frac{2}{2^{1/r}}. \]
    \item\label{item: L_p infty r} For $r=\infty$
    \[ \min\left\{\frac{2}{2^{1/p}}, \frac{2}{2^{1/q}}\right\}\leq \gamma(L_p(\mu) \oplus_\infty \Y)\leq \gamma^*(L_p(\mu) \oplus_\infty \Y)\leq \max\left\{ \frac{2}{2^{1/p}}, \gamma^*(\Y)\right\}. \]
\end{enumerate}
\end{theorem}

Let us explicitly point out the following particular case, where both $\gamma$ and $\gamma^*$ can be computed exactly: if $1\leq r\leq p\leq 2$,
\[ \gamma(L_p(\mu)\oplus_r \Y)= \gamma^*(L_p(\mu)\oplus_r \Y)= \frac{2}{2^{1/p}}. \]

\begin{proof} If $\kappa_0\leq \kappa$ is any infinite cardinal, $K(L_p(\mu);\kappa_0)\leq K(L_p(\mu))\leq \max\{2^{1/p}, 2^{1/q}\}$, where the second inequality follows from \cite{WellWill}*{Theorem 16.9} quoted above. The rest of the proof is identical to the one of \Cref{thm: gamma ell_p} and we omit repeating it.
\end{proof}

\begin{remark} In \cite{CLL}*{Problem 6.2} it is asked whether $(<\kappa)$-octahedral Banach spaces need contain $\ell_1(\kappa)$, which is then answered negatively in \cite{ACLLR}*{Remark 5.4}. The example given there is the Banach space $\X\coloneqq (\bigoplus_{k= 1}^\infty \ell_{p_k}(\omega_1))_{\ell_1}$. In fact, it is easy to check directly that $\X$ is $(<\omega_1)$-octahedral for every sequence $(p_k)_{k=1}^\infty \subseteq (1,\infty)$ such that $p_k\to 1$. On the other hand, $\X$ is WCG, whence it doesn't contain $\ell_1(\omega_1)$.

By means of \Cref{prop: L_p p-octa}, we can give an alternative, somewhat more natural, example. Indeed, if $\mu$ is a finite measure such that $L_1(\mu)$ is non-separable, $L_1(\mu)$ is $(<\omega_1)$-octahedral by \Cref{prop: L_p p-octa} and yet it doesn't contain $\ell_1(\omega_1)$, as it is WCG. This last fact also follows from a result due to Enflo and Rosenthal \cite{EnfloRosenthal}*{Theorem 2.1}, who proved that $\ell_p(\omega_1)$ is not isomorphic to a subspace of $L_p(\mu)$ for any finite measure $\mu$ and $p\in[1,\infty)$.
\end{remark}

In conclusion to this section, we briefly mention the case $p=\infty$ and show that every $L_\infty(\mu)$ space admits a lattice tiling by balls. In particular, $\gamma^*(L_\infty(\mu))=1$. The proof is immediate and it involves again the even integers grid.
\begin{proposition}\label{prop: L_infty} Every space $L_\infty(\mu)$ admits a lattice tiling by balls.
\end{proposition}

\begin{proof} Let $(\M,\Sigma, \mu)$ be any measure space and consider the set $L_\infty(\mu;2\ZZ)$ of $2\ZZ$-valued functions in $L_\infty(\mu)$ (more precisely, equivalence classes of functions having one $2\ZZ$-valued representative). Plainly, $L_\infty(\mu;2\ZZ)$ is $2$-separated and we shall show that it is $1$-dense. If $f\in L_\infty(\mu)$, the sets
\[ U_k\coloneqq\{ m\in\M\colon  2k-1\leq f(m)< 2k+1 \} \]
are measurable and $\{U_k\}_{k\in \ZZ}$ is a partition of $\M$; additionally, only finitely many $U_k$'s are non-empty. Thus, the function
\[ g\coloneqq \sum_{k\in \ZZ} 2k\cdot \bone_{U_k} \]
belongs to $L_\infty(\mu;2\ZZ)$ and $\|f-g\|\leq 1$.
\end{proof}

\begin{remark} It follows from a result due to Pe\l cz\'ynski and Sudakov \cite{Sudakov} that there are $L_\infty(\mu)$ spaces that are not $1$-injective (see, \emph{e.g.}, the explanation in \cite{JKS_PAMS}*{p.~4473}). Therefore, the result above is not a consequence of \Cref{thm: C(K)}.
\end{remark}

\section{Open problems}\label{sec: problem}
In this last section we highlight a selection of natural problems that arise from our results. As it should be apparent at this point, the constants $\gamma(\X)$ and $\gamma^*(\X)$ are still far from being well understood and there is a large area for further research. Here we just present a few possible directions. We begin with a couple of problems concerning specific Banach spaces.

\begin{problem}\label{probl: ell1 +2 R} What is the exact value of $\gamma(\ell_1\oplus_2\R)$? Does it coincide with $\gamma^*(\ell_1\oplus_2\R)$? Further, what are the values of $\gamma(\ell_1\oplus_2 \ell_2)$ and $\gamma^*(\ell_1\oplus_2 \ell_2)$?
\end{problem}

Recall that in \Cref{ex: ell1 +2 R} we showed that $\gamma(\ell_1\oplus_2\R)>1$. Further, $\gamma(\ell_1\oplus_2 \R)\leq \gamma^*(\ell_1\oplus_2 \R)\leq \sqrt{2}$, because $\ell_1\oplus_2\R$ is $\sqrt{2}$-isomorphic to $\ell_1$. The state for $\ell_1\oplus_2 \ell_2$ is similar. We have $\gamma(\ell_1\oplus_2 \ell_2)>1$ by \Cref{thm: gamma LUR}, as this space admits a LUR point. Further, $\ell_1\oplus_1\ell_2$ is octahedral and $\sqrt{2}$-isomorphic to $\ell_1\oplus_2 \ell_2$; thus, $\gamma^*(\ell_1\oplus_2 \ell_2)\leq \sqrt{2}$ by \Cref{thm: octahedral}.

\begin{problem} What are the values of $\gamma(\ell_1 \oplus_2 \ell_1)$ and $\gamma^*(\ell_1 \oplus_2 \ell_1)$?
\end{problem}

Note that $\ell_1 \oplus_2 \ell_1$ is not octahedral, nor its unit ball has LUR points. Therefore, the only information we have is that $\gamma^*(\ell_1 \oplus_2 \ell_1)\leq \sqrt{2}$, because its Banach--Mazur distance from $\ell_1$ is $\sqrt{2}$ (or because $\ell_1 \oplus_2 \ell_1$ is $\phi_2$-octahedral, by \Cref{prop: p-octa sum Y}). Further, notice that $\gamma^*(\ell_1 \oplus_\infty \ell_1)=1$ (\Cref{fact: gamma and d_BM}) and $\gamma^*(\ell_1 \oplus_1 \ell_1)=1$. Thus, $\gamma^*(\ell_1 \oplus_p \ell_1) \leq \sqrt{2}$ for all $p\in [1, \infty]$ and it would also be interesting to study the (continuous) function $p\mapsto \gamma^*(\ell_1 \oplus_p \ell_1)$; for instance, is it true that it attains its maximum in $p=2$?

We now pose a problem concerning function spaces, related to our results in \Cref{sec: gamma=1} and \Cref{sec: L_p}.
\begin{problem} 
Is it true that $\gamma^*(\C(\K))= 1$ for all (metrisable) compact topological spaces?

What are the values of $\gamma(L_p(\mu))$ and $\gamma^*(L_p(\mu))$ for $p\in (2,\infty)$?
\end{problem}

We now mention just one possible sample problem in renorming theory.
\begin{problem}\label{probl: renorm ell2} Is there a norm $\nn\cdot$ on $\ell_2$ such that $\gamma((\ell_2,\nn\cdot)) >\sqrt{2}$?
\end{problem}

As we mentioned in the Introduction, the isomorphic theory is essentially unexplored and this question is just one possible direction, motivated by the fact that $\ell_2$ can be renormed to obtain $\gamma^*((\ell_2,\nn\cdot)) =1$ (\Cref{prop: renorm ell_2}). The question should also be compared to the fact that $K((\ell_2,\nn\cdot)) \geq\sqrt{2}$ for all norms on $\ell_2$, \cite{Kottman}.

Finally, the last two problems we ask involve an arbitrary normed space.
\begin{problem}\label{probl: separable gamma=2} Is there a separable normed space $\X$ such that $\gamma(\X)= 2$?

What about a (separable) super-reflexive one?
\end{problem}

As we saw in \Cref{sec: gamma=2}, there exist normed spaces $\X$ with $\gamma(\X)= 2$; the argument given there could produce reflexive, or octahedral, examples, but only of density at least $\omega_\omega$ and not super-reflexive. Recall that, by \Cref{fact: finite-dim gamma}, there can't be finite-dimensional examples.

\begin{problem}\label{probl: gamma= gamma*} Is there a normed space $\X$ such that $\gamma(\X)\neq \gamma^*(\X)$?
\end{problem}

Differently from the previous problem, this one is also open for finite-dimensional spaces and we refer to the Introduction for some partial results and references. It is even conjectured, see \emph{e.g.}, \cite{Zong_sphere_book}*{Problem 11.5}, that $\gamma(\R^n)\neq \gamma^*(\R^n)$ for some large values of $n$. (Likewise, it was conjectured by Rogers that $\delta_n\neq \delta^*_n$ for some values of $n$, \cite{Rogers}*{p.~14}.) One way to answer the problem in $\R^n$ (by \Cref{fact: finite-dim gamma}) would be to show that $\gamma^*(\R^n)\geq 2$ for some $n\in \N$. On the other hand, it is not completely inconceivable that $\gamma(\R^n)\neq \gamma^*(\R^n)$ for some $n\in \N$, while $\gamma(\X)= \gamma^*(\X)$ for all infinite-dimensional normed spaces.

\bigskip
\noindent\textbf{Acknowledgements.} We are grateful to Gilles Godefroy, Johann Langemets, and Abraham Rueda Zoca for some remarks on the notion of $\phi$-octahedrality, and to Konrad Swanepoel and Chuanming Zong for some references concerning the simultaneous packing and covering constant. We also thank Piero Papini and Clemente Zanco for providing us with a copy of the papers \cites{CPZ, Rog84}.


\end{document}